\documentclass{colt2015} 

\usepackage{algorithm,algorithmic}
\usepackage{amssymb, multirow, paralist, color}
\newtheorem{thm}{Theorem}
\newtheorem{prop}{Proposition}
\newtheorem{cor}{Corollary}
\newtheorem{defn}{Definition}

\usepackage{enumitem}
\usepackage{booktabs}

\def \S {\mathbf{S}}

\def \R {\mathbb{R}}

\def \v {\mathbf{v}}

\def \x {\mathbf{x}}

\def \x {\mathbf{x}}

\def \a {\mathbf{a}}

\def \b {\mathbf{b}}

\def \y {\mathbf{y}}
\def \u {\mathbf{u}}

\def \y {\mathbf{y}}

\def \x {\mathbf{x}}

\def \u {\mathbf{u}}

\def \R {\mathbb{R}}
\def \S {\mathcal{S}}

\def \v {\mathbf{v}}

\def \a {\mathbf{a}}
\def \b {\mathbf{b}}

\begin{document}

\title[Adaptive Accelerated Gradient Converging Methods]{Adaptive Accelerated Gradient Converging Methods \\under H\"{o}lderian Error Bound Condition}

 \coltauthor{\Name{Mingrui Liu}\Email{mingrui-liu@uiowa.edu}\\
\Name{Tianbao Yang} \Email{tianbao-yang@uiowa.edu}\\
   \addr Department of Computer Science \\
    The University of Iowa, Iowa City, IA 52242 
}

\maketitle
\vspace*{-0.5in}
\begin{center}{First version: November 22, 2016}\end{center}

\begin{abstract}
Recent studies have shown that proximal gradient (PG) method and accelerated gradient method (APG) with restarting can enjoy a linear convergence under a weaker condition than strong convexity, namely a quadratic growth condition (QGC). However, the faster convergence of restarting APG method relies on the potentially unknown constant  in QGC to appropriately restart APG, which restricts its applicability. We address this issue by developing a novel adaptive gradient converging methods, i.e., leveraging  the magnitude of proximal gradient as a criterion for restart and termination. Our analysis extends to a much more general condition beyond the QGC, namely the H\"{o}lderian error bound (HEB) condition.   
{\it The key technique} for our development is a novel synthesis of  {\it adaptive regularization and a conditional restarting scheme}, which extends previous work focusing on strongly convex problems to a much broader family of problems. Furthermore, we demonstrate that our results have important implication and applications in machine learning: (i) if the  objective function is coercive and semi-algebraic, PG's convergence speed is essentially $o(\frac{1}{t})$, where $t$ is the total number of iterations; (ii) if the objective function consists of an $\ell_1$, $\ell_\infty$, $\ell_{1,\infty}$, or huber  norm regularization and a convex smooth piecewise quadratic loss (e.g., squares loss, squared hinge loss and huber loss), 
the proposed algorithm is parameter-free and enjoys a {\it faster linear convergence} than PG without any other assumptions  (e.g., restricted eigen-value condition).   It is notable that  our linear convergence results for the aforementioned problems  are global instead of local.  To the best of our knowledge, these improved results are the first shown in this work.
\end{abstract}

\section{Introduction}

We consider the following smooth optimization problem:
\begin{align}\label{eqn:smt}
\min_{\x\in\R^d}f(\x),
\end{align}
where $f(\x)$ is a continuously differential convex function, whose gradient is $L$-Lipschitz continuous. More generally, we also tackle the following 
composite optimization:
\begin{align}\label{eqn:opt}
	\min_{\x\in\R^d} F(\x)\triangleq f(\x) + g(\x),
\end{align}
where $g(\x)$ is a proper lower semi-continuous convex function and $f(\x)$ is a continuously differentiable convex function, whose gradient is $L$-Lipschitz continuous. 
The above problem has been studied extensively in literature and many algorithms have been developed with convergence guarantee. In particular, by employing the proximal mapping associated with $g(\x)$, i.e., 
\begin{align}\label{eqn:prox}
	P_{\eta g}(\u) = \arg\min_{\x\in\R^d}\frac{1}{2}\|\x - \u\|_2^2+ \eta g(\x),
\end{align}
proximal gradient (PG) and accelerated proximal gradient (APG) methods have been developed for solving~(\ref{eqn:opt}) with $O(1/\epsilon)$ and $O(1/\sqrt{\epsilon})$~\footnote{For the moment, we neglect the constant factor.} iteration complexities for finding an $\epsilon$-optimal solution. 
When either $f(\x)$ or $g(\x)$ is strongly convex, both PG and APG can enjoy a linear convergence, i.e., the iteration complexity is improved to be $O(\log(1/\epsilon))$. 
\begin{table*}[t]
	\caption{Summary of iteration complexities  in this work under the HEB condition with $\theta\in(0,1/2]$, where $G(\x)$ denotes the proximal gradient, $\mathcal C(1/\epsilon^\alpha) = \max(1/\epsilon^\alpha, \log(1/\epsilon))$ and $\widetilde O(\cdot)$ suppresses a logarithmic term. If $\theta>1/2$, all algorithms can converge with finite steps of proximal mapping. rAPG stands for restarting APG. } 
	\centering
	\label{tab:data}
	\begin{tabular}{l|lll}
		\toprule
		algo.&PG &rAPG&adaAGC\\
		\midrule
		$F(\x) - F_*\leq \epsilon$ &$O\left(c^2L\mathcal C\left(\frac{1}{\epsilon^{1-2\theta}}\right)\right)$&$O\left(c\sqrt{L}\mathcal C\left(\frac{1}{\epsilon^{1/2-\theta}}\right)\right)$&--\\ 
		\midrule
		$\|G(\x)\|_2\leq \epsilon$ & $O\left(c^{\frac{1}{1-\theta}}L\mathcal C\left(\frac{1}{\epsilon^{\frac{1-2\theta}{1-\theta}}}\right)\right)$&--&$\widetilde O\left(c^{\frac{1}{2(1-\theta)}}\sqrt{L}\mathcal C\left(\frac{1}{\epsilon^{\frac{1-2\theta}{2(1-\theta)}}}\right)\right)$\\
		\midrule
		requires $\theta$&No&Yes&Yes\\
		\midrule
		requires $c$&No&Yes&No\\
		\bottomrule
	\end{tabular}
\end{table*}

Recently, a wave of study is to generalize the linear convergence to problems without strong convexity but under certain structured condition of the objective function or more generally a quadratic growth condition~\citep{DBLP:conf/nips/HouZSL13,DBLP:conf/icml/ZhouZS15,DBLP:journals/corr/So13,DBLP:journals/jmlr/WangL14,DBLP:journals/corr/GongY14,DBLP:journals/corr/ZhouS15a,arxiv:1510.08234,DBLP:journals/corr/nesterov16linearnon,DBLP:conf/pkdd/KarimiNS16,DBLP:journals/corr/abs/1606.00269,Drusvyatskiy16a}. Earlier work along the line dates back to~\citep{Luo:1992a,Luo:1992b,Luo:1993}. An example of the structured condition is such that $f(\x) = h(A\x)$ where $h(\cdot)$ is strongly convex function and $\nabla h(\x)$ is Lipschitz continuous on any compact set, and $g(\x)$ is a polyhedral function. Under such a structured condition, a local error bound condition can be established~\citep{Luo:1992a,Luo:1992b,Luo:1993}, which renders an asymptotic (local) linear convergence for the proximal gradient method. A quadratic growth condition (QGC) prescribes that  the objective function satisfies for any $\x\in\R^d$~\footnote{It can be relaxed to a fixed domain as done in this work.}:
$\frac{\alpha}{2}\|\x - \x_*\|_2^2 \leq F(\x) - F(\x_*)$,
where $\x_*$ denotes a closest point to $\x$ in the optimal set. Under such a quadratic growth condition, several recent studies have established the linear convergence of PG, APG and many other algorithms (e.g., coordinate descent methods)~\citep{arxiv:1510.08234,DBLP:journals/corr/nesterov16linearnon,Drusvyatskiy16a,DBLP:conf/pkdd/KarimiNS16,DBLP:journals/corr/abs/1606.00269}. A notable result is that PG enjoys an iteration complexity of $O(\frac{L}{\alpha}\log(1/\epsilon))$ without knowing  the value  of $\alpha$, while a restarting version of APG studied in~\cite{DBLP:journals/corr/nesterov16linearnon} enjoys an improved iteration complexity of $O(\sqrt{\frac{L}{\alpha}}\log(1/\epsilon))$ hinging  on the value of $\alpha$ to appropriately restart APG periodically. Other equivalent conditions or more restricted conditions are also considered in several studies to show the linear convergence of (proximal) gradient method and other methods~\citep{DBLP:conf/pkdd/KarimiNS16,DBLP:journals/corr/nesterov16linearnon,DBLP:journals/corr/abs/1606.00269,Zhang2016}. 

In this paper, we extend this line of work to a more general error bound condition, i.e., the H\"{o}lderian error bound (HEB) condition on a compact sublevel set $\mathcal S_\xi=\{\x\in\R^d: F(\x) - F(\x_*)\leq \xi\}$: there exists $\theta\in(0,1]$ and $0<c< \infty$ such that 
\begin{align}\label{eqn:HEB}
	&\|\x - \x_*\|_2\leq c(F(\x) - F(\x_*))^{\theta},\:\forall\x \in\mathcal S_\xi. 
\end{align}
Note that when $\theta=1/2$ and $c = \sqrt{1/\alpha}$, the HEB reduces to the QGC. In the sequel, we will refer to $C = Lc^2$ as condition number of the problem. It is worth mentioning that \citet{arxiv:1510.08234} considered the same condition or an equivalent  Kurdyka - \L ojasiewicz inequality but they only focused on descent methods that bear a sufficient decrease condition for each update consequentially  excluding APG. In addition, they do not provide explicit iteration complexity under the general HEB condition.

As a warm-up and motivation, we will first present a straightforward analysis to show that PG is automatically adaptive and APG can be made adaptive to the HEB by restarting.  In particular if $F(\x)$ satisfies a HEB condition on the initial sublevel set, PG has an iteration complexity of $O(\max(\frac{C}{\epsilon^{1-2\theta}}, C\log(\frac{1}{\epsilon})))$~\footnote{When $\theta>1/2$, all algorithms can converge in finite steps.}, and restarting APG enjoys an iteration complexity of $O(\max(\frac{\sqrt{C}}{\epsilon^{1/2-\theta}}, \sqrt{C}\log(\frac{1}{\epsilon})))$ for the convergence of objective value, where $C = Lc^2$ is the condition number. These two results resemble but generalize recent works that establish linear convergence of PG and restarting APG under the QGC - a special case of HEB.  Although enjoying faster convergence, restarting APG has some caveats: (i) it requires the knowledge of  constant $c$  in HEB to restart APG, which is usually difficult to compute or estimate; (ii) there lacks an appropriate machinery to terminate the algorithm. 
In this paper, we make nontrivial contributions to obtain faster convergence of the proximal gradient's norm under the HEB condition by developing an adaptive accelerated gradient converging method.  

The main results of this paper are summarized in Table 1. In summary the contributions of this paper are:
\begin{itemize}
	\item We extend the analysis of PG and restarting APG under the quadratic growth condition to more general HEB condition, and establish the adaptive iteration complexities of both algorithms. 
\item To enjoy faster convergence of restarting APG and to eliminate the algorithmic dependence on the unknown parameter $c$, we propose and analyze an adaptive accelerated gradient converging (adaAGC) method. 
\end{itemize}
{The developed algorithms and theory have important implication and applications in machine learning.} Firstly, if the considered objective function is also coercive and semi-algebraic  (e.g., a norm regularized problem in machine learning with a semi-algebraic loss function), then PG's convergence speed is essentially $o(1/t)$ instead of $O(1/t)$, where $t$ is the total number of iterations.   Secondly, for solving $\ell_1$, $\ell_\infty$ or $\ell_{1,\infty}$ regularized smooth loss minimization problems  including least-squares loss, squared hinge loss and huber loss, the proposed adaAGC method enjoys a linear convergence and a square root dependence on the ``condition" number. In contrast to previous work, the proposed algorithm is parameter free and does not rely on any restricted  conditions (e.g., the restricted eigen-value conditions).  

\section{Related Work}
At first, we review some related work for solving the problem (\ref{eqn:smt}) and (\ref{eqn:opt}). In Nesterov's seminal work \citep{nesterov1983method,RePEc:cor:louvco:2007076}, the accelerated (proximal) gradient (APG) method were proposed for (composite) smooth optimization problems, enjoying $O(1/\sqrt{\epsilon})$ iteration complexity for achieving a $\epsilon$-optimal solution. When the objective is also strongly convex, APG can converge to the optimal solution linearly with an appropiate step size depending on the strong convexity modulus, which enjoys $O(\log{(1/\epsilon)})$ iteration complexity.

To address the issue of unknown strong convexity modulus for some problems, several restarting schemes were developed. \citet{RePEc:cor:louvco:2007076} proposed a restarting scheme for the APG method to approximate the unknown strongly convexity parameter and achieved a linear convergence rate. \citet{DBLP:conf/icml/LinX14} proposed an adaptive APG method which employs the restart and line search technique to automatically estimate the strong convexity parameter.
\citet{o2015adaptive} proposed an heuristic approach to adaptively restart accelerated gradient schemes and showed good experimental results. Nevertheless,  they provide no theoretical guarantee of their proposed heuristic approach. In contrast to these work, we do not assume any strong convexity or restricted strong convexity for sparse learning.  It was brought to our attention that  a recent work~\citep{arxiv:1609.07358} considered QGC and proposed restarted accelerated gradient and coordinate descent methods, including APG, FISTA and the accelerated proximal coordinate descent method (APPROX). The difference from their restarting scheme for APG and the restarting schemes in~\citep{RePEc:cor:louvco:2007076, DBLP:conf/icml/LinX14, o2015adaptive} and the present work is that their restart doest not involve evaluation of the gradient or the objective value but rather depends on a restarting frequency parameter and a convex combination parameter for computing the restarting solution, which can be set based on a rough estimate of the strong convexity parameter. As a result, their linear convergence (established for distance of solutions to the optimal set) heavily depends on the rough estimate of the strong convexity parameter. 

Leveraging error bound conditions dates back to~\citep{Luo:1992a,Luo:1992b,Luo:1993}, which  employed the error bound condition to establish the asymptotic (local) linear convergence for feasible descent methods. 
Luo \& Tseng' bounds the distance of a local solution to the optmal  set by the norm of  proximal gradient. 
 Several recent work~\citep{DBLP:conf/nips/HouZSL13,DBLP:conf/icml/ZhouZS15,DBLP:journals/corr/So13} have considered Luo \& Pseng's error bound condition for more problems in machine learning  
and established local linear convergence for proximal gradient methods. \citet{wang2014iteration} established a global error bound version of Luo \& Pseng's condition for a family of problems in machine learning (e.g., the dual formulation of SVM), 
and provided the global linear convergence for a series of algorithms, including cyclic coordinate descent methods for solving dual support vector machine. 
Note that the H\"{o}lderian error bound~\citep{arxiv:1510.08234} used in our analysis is different from Luo \& Pseng's condition, and is actually more general.   
\citet{arxiv:1510.08234} established the equivalence of HEB and Kurdyka-\L ojasiewicz (KL) inequality and showed how to derive lower computational complexity via employing KL inequality. As a special case of H\"{o}lderian error bound condition, quadratic growth condition (QGC) has been considered in several recent work for deriving linear convergence. \citet{DBLP:journals/corr/GongY14} established linear convergence of proximal variance-reduced gradient (Prox-SVRG) algorithm  under QGC. \citet{DBLP:journals/corr/nesterov16linearnon} showed that QGC is one of the relaxations of strong convexity conditions, which can still guarantee the linear convergence for several first order methods, including projected gradient, fast gradient and feasible descent methods. \citet{Drusvyatskiy16a} also showed that proximal gradient algorithm achieved the linear convergence under QGC. There also exist other  conditions  (stronger than or equivalent to QGC)  that can help achieve linear convergence rate. For example, \citet{DBLP:conf/pkdd/KarimiNS16} showed that the Polyak-\L ojasiewicz (PL) inequality suffices to guarantee a global linear convergence for (proximal) gradient descent methods. \citet{DBLP:journals/corr/abs/1606.00269} summarized different sufficient conditions which are capable of deriving linear convergence, and discussed their relationships.



\section{Notations and Preliminaries}
In this section, we present some notations and preliminaries. In the sequel, we let $\|\cdot\|_p$ ($p\geq 1$) denote the $p$-norm of a vector. A function $g(\x): \R^d\rightarrow ]-\infty, \infty] $ is a proper function if $g(\x)<+\infty$ for at least one $\x$ and $g(\x)>-\infty$ for all $\x$. $g(\x)$ is lower semi-continuous at a point $\x_0$ if $\lim\inf_{\x\rightarrow \x_0}g(\x) = g(\x_0)$. A function $F(\x)$ is coercive if and only if $F(\x)\rightarrow \infty$ as $\|\x\|_2\rightarrow\infty$.  

A subset $S\subset\R^d$ is a real semi-algebraic set if there exists a finite number of real polynomial functions $g_{ij}, h_{ij}:\R^d\rightarrow\R$ such that 
\[
S = \cup_{j=1}^p \cap_{i=1}^q\{\u\in\R^d; g_{ij}(\u)=0 \text{ and } h_{ij}(\u)\leq 0\}. 
\] 
A function $F(\x)$ is semi-algebraic if its graph $\{(\u, s)\in\R^{d+1}: F(\u)=s\}$ is a semi-algebraic set. 

Denote by $\mathbb N$ the set of all positive integers. A function $h(\x)$ is a real polynomial  if there exists $r\in\mathbb N$ such that $h(\x) = \sum_{0\leq |\alpha|\leq r}\lambda_\alpha \x^\alpha$, where $\lambda_\alpha\in\R$ and $\x^\alpha =x_1^{\alpha_1}\ldots x_d^{\alpha_d}$, $\alpha_j\in\mathbb N\cup \{0\}$, $|\alpha| =\sum_{j=1}^d\alpha_j$ and $r$ is referred to as the degree of $h(\x)$. A continuous function $f(\x)$ is said to be a piecewise convex polynomial if there exist finitely many polyhedra $P_1,\ldots, P_k$ with $\cup_{j=1}^kP_j = \R^n$ such that the restriction of $f$ on each $P_j$ is a convex polynomial. Let $f_j$ be the restriction of $f$ on $P_j$. The degree of a piecewise convex polynomial function $f$ denoted by $deg(f)$ is the maximum of the degree of each $f_j$. If $deg(f)=2$, the function is referred to as a piecewise convex quadratic function. Note that a piecewise convex polynomial function is not necessarily a convex function~\citep{DBLP:journals/mp/Li13}.

A function $f(\x)$ is $L$-smooth w.r.t $\|\cdot\|_2$ if it is differentiable and has a Lipschitz continuous gradient with the Lipschitz constant $L$, i.e.,  $\|\nabla f(\x) - \nabla f(\y)\|_2 \leq L\|\x - \y\|_2, \forall \x, \y$.
Let $\partial g(\x)$ denote the subdifferential of $g$ at $\x$, i.e., 
\begin{equation*}
	\partial g(\x) = \{\u\in\R^d: g(\y)\geq g(\x) + \u^{\top}(\y - \x), \forall\y\}.
\end{equation*}
Denote by $\|\partial g(\x)\|_2 = \min_{\u\in\partial g(\x)}\|\u\|_2$.  
A function $g(\x)$ is $\alpha$-strongly convex w.r.t $\|\cdot\|_2$ if it satisfies for any $\u\in\partial g(\y)$ such that
$g(\x)\geq g(\y) + \u^{\top}(\x - \y) + \frac{\alpha}{2}\|\x - \y\|_2^2, \forall \x, \y$. 

Denote by $\eta>0$ a positive scalar, and  let $P_{\eta g}$ be the proximal mapping associated with $\eta g(\cdot)$ defined in~(\ref{eqn:prox}). 
Given an objective function $F(\x) = f(\x) + g(\x)$, where $f(\x)$ is $L$-smooth and $g(\x)$ is a simple non-smooth function, define a {\bf proximal gradient} $G_\eta(\x)$ as: 
\begin{align*}
	G_\eta(\x) = \frac{1}{\eta}(\x - \x^+_\eta), \text{ where }\x^+_\eta = P_{\eta g}(\x - \eta\nabla f(\x))
\end{align*}
When $g(\x)=0$, we have $G_\eta(\x) = \nabla f(\x)$, i.e., the proximal gradient is the gradient. It is known that  $\x$ is an optimal solution iff $G_\eta(\x)=0$. If $\eta = 1/L$, for simplicity we denote by $G(\x) = G_{1/L}(\x)$ and $\x^+ = P_{g/L}(\x - \nabla f(\x)/L)$.  Below, we give several technical propositions related to $G_\eta(\x)$ and the proximal gradient update.
 \begin{prop}\citep{RePEc:cor:louvco:2007076}
 	\label{decreaseNesterov}
 Given $\x$, $\|G_\eta(\x)\|_2$ is a monotonically decreasing function of $\eta$. 
  \end{prop}

 \begin{prop}\citep{Beck:2009:FIS:1658360.1658364}\label{lem:2}
 Let $F(\x) = f(\x)  + g(\x)$. Assume $f(\x)$ is $L$-smooth. For any $\x, \y$ and $\eta \leq 1/L$, we have
\begin{align}\label{eqn:key}
F(\y^+_\eta)\leq F(\x)  + G_\eta(\y)^{\top}(\y - \x)  - \frac{\eta}{2}\|G_\eta(\y)\|_2^2.
\end{align}
 \end{prop}
The following corollary is useful for our analysis. 
\begin{cor}\label{cor:1}
	Let $F(\x) = f(\x)  + g(\x)$. Assume $f(\x)$ is $L$-smooth. For any $\x, \y$ and $0<\eta \leq 1/L$, we have
	\begin{equation}\label{eqn:key2}
		\frac{\eta}{2}\|G_\eta(\y)\|_2^2\leq F(\y) - F(\y^+_\eta)\leq F(\y) - \min_{\x}F(\x).
	\end{equation}
\end{cor}
\textbf{Remark:} The proof of Corollary \ref{cor:1} is immediate by employing the convexity of $F$ and Proposition \ref{lem:2}.

Let $F_*$ denote the optimal objective value to $\min_{\x\in\R^d}F(\x)$ and $\Omega_*$ denote the optimal set. Denote by $\S_{\xi}=\{\x: F(\x) - F_*\leq \xi\}$ the $\xi$-sublevel set of $F(\x)$. Let $D(\x, \Omega) = \min_{\y\in\Omega}\|\x - \y\|_2$.

The proximal gradient (PG) method solves the problem~(\ref{eqn:opt}) by the 
update  
\begin{align}\label{eqn:pg}
\x_{t+1}=P_{\eta g}(\x_t - \eta \nabla f(\x_t)), %
\end{align}
with $\eta\leq 1/L$ starting from some initial solution $\x_1\in\R^d$. It can be shown that PG has an iteration complexity of $O(\frac{LD(\x_1,\Omega_*)^2}{\epsilon})$. 
The convergence guarantee of PG is presented in the following proposition. 
\begin{prop}\citep{opac-b1104789}\label{prop:pg}
Let~(\ref{eqn:pg}) run for $t=1,\ldots, T$ with $\eta\leq 1/L$, we have
\[
F(\x_{T+1}) - F_*\leq \frac{D(\x_1, \Omega_*)^2}{2\eta T}.
\]
\end{prop}
Based on the above proposition, one can deduce that PG has an iteration complexity of $O(\frac{LD(\x_1,\Omega_*)^2}{\epsilon})$. 
Nevertheless, accelerated proximal gradient (APG) converges faster than PG. There are many variants of APG in literature~\citep{citeulike:6604666}.  The simplest variant adopts the following update 
\begin{equation}
\label{eqn:APG}
	\begin{cases}
		\y_{t}  = \x_t + \beta_t (\x_t - \x_{t-1}),\\
		 \x_{t+1} = P_{\eta g}(\y_t - \eta\nabla f(\y_t)),
	\end{cases}
\end{equation}
where $\eta\leq 1/L$ and $\beta_t = \frac{t-1}{t+2}$. APG enjoys an iteration complexity of $O(\frac{\sqrt{L}D(\x_1,\Omega_*)}{\sqrt{\epsilon}})$~\citep{citeulike:6604666}. 
The convergence guarantee of APG is presented in the following proposition. 
\begin{prop}\citep{citeulike:6604666}\label{prop:apg}
Let~(\ref{eqn:APG}) run for $t=1,\ldots, T$ with $\eta\leq 1/L$ and $\x_0 =\x_1$, we have
\[
F(\x_{T+1}) - F_*\leq \frac{2D(\x_1, \Omega_*)^2}{\eta (T+1)^2}.
\]
\end{prop}

Based on the above proposition, one can deduce that APG has an iteration complexity of $O(\frac{\sqrt{L}D(\x_1,\Omega_*)}{\sqrt{\epsilon}})$. 

Furthermore, if $f(\x)$ is both $L$-smooth and $\alpha$-strongly convex, one can set $\beta_t = \frac{\sqrt{L}-\sqrt{\alpha}}{\sqrt{L}+\sqrt{\alpha}}$  and deduce a linear convergence~\citep{DBLP:conf/icml/LinX14} with a better dependence on the condition number than that of PG. 
\begin{prop}\citep{DBLP:conf/icml/LinX14}\label{prop:3}
Assume $f(\x)$ is $L$-smooth and $\alpha$-strongly convex. Let~(\ref{eqn:APG}) run for $t=1,\ldots, T$ with $\eta=1/L$, $\beta_t = \frac{\sqrt{L}-\sqrt{\alpha}}{\sqrt{L}+\sqrt{\alpha}}$  and $\x_0 =\x_1$, we have for any $\x$
\[
F(\x_{T+1}) - F(\x)\leq \left(1 - \sqrt{\frac{\alpha}{L}}\right)^{T}\left[F(\x_0) - F(\x) + \frac{\alpha}{2}\|\x_0 - \x\|_2^2\right].
\]
\end{prop}
If $\phi(\x)$ is $\alpha$-strongly convex and $f(\x)$ is $L$-smooth, \citet{RePEc:cor:louvco:2007076} proposed a different variant based on dual averaging, which is referred to accelerated dual gradient  (ADG) method and will be useful for our develeopment.  The key steps are presented in Algorithm~\ref{alg:0}. The convergence guarantee of ADG is given the following proposition. 
\begin{prop}\label{prop:adg}\citep{RePEc:cor:louvco:2007076}
	\label{Nesterov2007}
	Assume $f(\x)$ is $L$-smooth and $g(\x)$ is $\alpha$-strongly convex. Let Algorithm~\ref{alg:0} run for $t=0,\ldots, T$. Then for any $\x$ we have 
	\[
	F(\x_{T+1}) - F(\x)\leq \frac{L}{2}\|\x_0 - \x\|_2^2\left(\frac{1}{1+\sqrt{\alpha/2L}}\right)^{2T}.
	\]
\end{prop}
\begin{algorithm}[t]
	\caption{ADG} \label{alg:0}
	\begin{algorithmic}[1]
		\STATE $\x_0 \in\Omega$, $A_0=0$, $\v_0 =\x_0$
		\FOR{$t = 0, \ldots, T$}
		\STATE Find $a_{t+1}$ from quadratic equation $\frac{a^2}{A_t+a} = 2\frac{1+\alpha A_t}{L}$
		\STATE Set $A_{t+1} = A_t + a_{t+1}$
		\STATE Set $\y_t = \frac{A_t}{A_{t+1}}\x_t + \frac{a_{t+1}}{A_{t+1}}\v_t$
		\STATE Compute $\x_{t+1} =P_{g/L}(\y_{t} - \nabla f(\y_t)/L) $
		\STATE Compute $\v_{t+1} = \arg\min_{\x}\sum_{\tau=1}^{t+1}a_\tau\nabla f(\x_\tau)^{\top}\x + A_{t+1}g(\x) + \frac{1}{2}\|\x - \x_0\|_2^2$
		\ENDFOR
	\end{algorithmic}
\end{algorithm}

\subsection{H\"{o}lderian error bound (HEB) condition}

\begin{defn}[H\"{o}lderian error bound (HEB)]\label{def:1}
	A function $F(\x)$ is said to satisfy a HEB condition on the $\xi$-sublevel set if there exist $\theta\in(0,1]$ and $0<c<\infty$ such that for any $\x\in\mathcal S_{\xi}$
	\begin{align}\label{eqn:leb}
		dist(\x, \Omega_*)\leq c(F(\x) - F_*)^{\theta},
	\end{align}
	where $\Omega_*$ denotes the optimal set of $\min_{\x\in\R^d}F(\x)$. 
\end{defn}

The HEB condition is closely related to the  \L ojasiewicz inequality or more generally Kurdyka- \L ojasiewicz (KL) inequality in real algebraic geometry. It has been shown that when functions are semi-algebraic and continuous, the above inequality is known to hold on any compact set~\citep{arxiv:1510.08234}. We refer the readers to~\citep{arxiv:1510.08234} for more discussions on HEB and KL inequalities. 

In the remainder of this section, we will review some previous results to demonstrate that HEB is a generic condition that holds for a broad family of problems of interest. The following proposition states that any proper, coercive, convex, lower-semicontinuous and semi-algebraic functions satisfy the HEB condition. 
\begin{prop}\citep{arxiv:1510.08234}\label{prop:5}
	Let $F(\x)$ be a proper, coercive, convex, lower semicontinuous and semi-algebraic function. Then there exists $\theta\in(0,1]$ and $0<c<\infty$ such that $F(\x)$ satisfies the HEB on any $\xi$-sublevel set. 
\end{prop}
 {\bf Example:} Most optimization problems in machine learning  with an objective that consists of an empirical loss that is semi-algebraic (e.g., hinge loss, squared hinge loss, absolute loss, square loss) and a norm regularization $\|\cdot\|_p$ ($p\geq 1$ is a rational) or a norm constraint  are proper, coercive, lower semicontinuous and semi-algebraic functions. 

Next two propositions exhibit the value $\theta$ for piecewise convex quadratic functions and  piecewise convex polynomial functions. 
\begin{prop}\citep{DBLP:journals/mp/Li13}\label{prop:quadratic}
	Let $F(\x)$ be a piecewise convex quadratic function on $\R^d$. Suppose $F(\x)$ is convex. Then for any $\xi>0$, there exists $0<c<\infty$ such that 
	\[
	D(\x, \Omega_*)\leq c(F(\x) - F_*)^{1/2}, \forall \x\in\S_{\xi}.
	\] 
\end{prop}

Many problems in machine learning  are piecewise convex quadratic functions, which will be discussed more in Section~\ref{sec:app}.  
\begin{prop}\citep{DBLP:journals/mp/Li13}\label{prop:polyHEB}
	Let $F(\x)$ be a piecewise convex polynomial function on $\R^d$. Suppose $F(\x)$ is convex. Then for any $\xi>0$, there exists $c>0$ such that 
	\[
	D(\x, \Omega_*)\leq c(F(\x) - F_*)^{\frac{1}{(deg(F)-1)^d+1}}, \forall \x\in\S_{\xi}.
	\] 
\end{prop}
Indeed, for 
a polyhedral constrained convex polynomial, we can have a tighter result, as show below. 
\begin{prop}\citep{doi:10.1137/070689838}
	Let $F(\x)$ be a convex polynomial function on $\R^d$ with degree $m$. 
	If $P\subset\R^d$ is a polyhedral set, then the problem $\min_{\x\in P}F(\x)$ admits a global error bound: $\forall \x\in P$ there exists $0<c<\infty$ such that 
	\begin{equation}\label{eqn:geb}
		D(\x, \Omega_*)\leq c\left[(F(\x) - F_*) + (F(\x) - F_*)^{\frac{1}{m}}\right], 
	\end{equation}
\end{prop}

From the global error bound~(\ref{eqn:geb}), one can easily derive the H\"{o}lderian error bound condition~(\ref{eqn:HEB}). For an example, we can consider an $\ell_1$ constrained $\ell_p$ norm regression~\citep{doi:10.1080/03610928308828618}: 
\begin{align}\label{eqn:L1LP}
	\min_{\|\x\|_1\leq s}F(\x)\triangleq \frac{1}{n}\sum_{i=1}^n(\a_i^{\top}\x - b_i)^p,\quad p\in 2\mathbb N
\end{align}
which satisfies the HEB condition~(\ref{eqn:HEB}) with $\theta=\frac{1}{p}$. 

Many previous papers have considered a family of structured smooth composite problems:
\begin{align}\label{eqn:st}
	\min_{\x\in\R^d}F(\x) = h(A\x) + g(\x)
\end{align}
where $g(\x)$ is a polyhedral function and $h(\cdot)$ is a smooth and  strongly convex function on any compact set.  Suppose the optimal set of the above problem is non-empty and compact (e.g., the function is coercive) so is the sublevel set $\S_\xi$, it can been shown that such a function satisfies HEB with $\theta=1/2$ on any sublevel set $\S_\xi$.  Examples of $h(\u)$ include logistic loss $h(\u) = \sum_i\log(1+\exp(-u_i))$.
\begin{prop}\citep[Theorem 4.3]{DBLP:journals/corr/nesterov16linearnon}\label{prop:log}
	Suppose the optimal set of~(\ref{eqn:st}) is non-empty and compact, $g(\x)$ is a polyhedral function and $h(\cdot)$ is a smooth and  strongly convex function on any compact set. Then $F(\x)$ satisfies the HEB on any sublevel set $\S_\xi$ with $\theta=1/2$ for $\xi>0$.  
\end{prop}

Finally, we note that there exist problems that admit HEB with $\theta>1/2$. A trivial example is given by $F(\x) = \frac{1}{2}\|\x\|_2^2 + \|\x\|_p^p$ with $p\in[1,2)$, which satisfies HEB with $\theta=1/p\in(1/2, 1]$. An interesting non-trivial family of problems is that $f(\x)=0$ and $g(\x)$ is a piece-wise linear functions according to Proposition~\ref{prop:polyHEB}. PG or APG applied to such family of problems is closely related to proximal point algorithm~\citep{citeulike:9472207}. Explorations of such algorithmic connection is not the focus of this paper.  
%
%
%

\section{PG and restarting APG under HEB}
\begin{algorithm}[t]
	\caption{PG} \label{alg:1}
	\begin{algorithmic}[1]
		\STATE \textbf{Input}:
		$\x_0 \in\Omega$
		\FOR{$\tau=1,\ldots, t$}
		\STATE $\x_{\tau+1} =P_{g/L}(\x_{\tau} - \nabla f(\x_\tau)/L) $
		\ENDFOR
		
		\STATE Option I:  return $\x_{t+1}$
		\STATE Option II:  return $\x_k$ s.t. $G(\x_k)=\min_{\tau}\|G(\x_\tau)\|_2$
	\end{algorithmic}
\end{algorithm}

As a warm-up and motivation of the major contribution presented in next section, we present a convergence result of PG and a restarting  APG under the HEB condition. 
  We first  present a result of PG as shown in Algorithm~\ref{alg:1}. 


\begin{thm}
	\label{thm:1}
	Suppose $F(\x_0) - F_*\leq \epsilon_0$ and $F(\x)$ satisfies HEB on $\S_{\epsilon_0}$. 
	The iteration complexity of PG (with option I) for achieving $F(\x_t) - F_*\leq \epsilon$ is $O(c^2L\epsilon_0^{2\theta -1})$ if $\theta>1/2$, and is $O(\max\{\frac{c^2L}{\epsilon^{1 - 2\theta}}, c^2L\log(\frac{\epsilon_0}{\epsilon})\})$ if $\theta\leq 1/2$. 
\end{thm}
\begin{proof}
	Divide the whole FOR loop of the Algorithm \ref{alg:1} into $K$ stages, denote $t_k$ by the number of iterations in the $k$-th stage, and denote $\x_k$ by the updated $\x$ at the end of the $k$-th stage, where $k=1,\ldots K$. Define $\epsilon_k:=\frac{\epsilon_0}{2^k}$.
	
	Choose $t_k=\lceil{c^2L\epsilon_{k-1}^{2\theta-1}}\rceil$, and we will prove $F(\x_k)-F_*\leq\epsilon_k$ by induction. Suppose $F(\x_{k-1})-F_*\leq\epsilon_{k-1}$, we have $\x_{k-1}\in\S_{\epsilon_0}$. According to Proposition \ref{prop:pg}, at the $k$-th stage, we have
	$$F(\x_k)-F_*\leq\frac{L\Vert \x_{k-1}-\x_{k-1}^{*}\Vert_2^2}{2t_k},$$
	where $\x_{k-1}^{*}\in\Omega_{*}$, the closest point to $\x_{k-1}$ in the optimal set. By the HEB condition, we have
	$$F(\x_k)-F_*\leq\frac{c^2L\epsilon_{k-1}^{2\theta}}{2t_k}.$$
	Since $t_k\geq c^2L\epsilon_{k-1}^{2\theta-1}$, we have $F(\x_k)-F_*\leq\epsilon_k$. The total number of iterations is
	$$\sum_{k=1}^{K}t_k\leq O(c^2L\sum_{k=1}^{K}\epsilon_{k-1}^{2\theta-1}).$$
	From the above analysis, we see that after each stage, the optimality gap decreases by half, so taking $K=\lceil \log_2{\frac{\epsilon_0}{\epsilon}}\rceil$ guarantees $F(\x_k)-F_*\leq\epsilon$.
	
	If $\theta>1/2$, the iteration complexity is $O(c^2L\epsilon_0^{2\theta-1})$. If $\theta=1/2$, the iteration complexity is $O(c^2L\log{\frac{\epsilon_0}{\epsilon}})$. If $\theta<1/2$, the iteration complexity is 
	$$\sum_{k=1}^{K}t_k\leq O(c^2 L\sum_{k=1}^{K}(\frac{\epsilon_0}{2^{k-1}})^{2\theta-1})=O(c^2L/\epsilon^{1-2\theta}).$$
\end{proof}

\begin{algorithm}[t]
	\caption{restarting APG (rAPG)} \label{alg:2}
	\begin{algorithmic}[1]
		\STATE \textbf{Input}: the number of stages $K$ and $\x_0 \in\Omega$
		\FOR{$k=1,\ldots, K$}
		\STATE Set $\y^k_{1} = \x_{k-1}$ and $\x^k_1 = \x_{k-1}$
		\FOR{$\tau = 1, \ldots, t_k$}
		\STATE Update $\x^k_{\tau+1} =P_{g/L}(\y^k_{\tau} - \nabla f(\y^k_\tau)/L) $\\
		\STATE Update $\y^k_{\tau+1} = \x^k_{\tau+1} + \frac{\tau }{\tau+3 }(\x^k_{\tau+1} - \x^k_{\tau})$
		\ENDFOR
		\STATE Let $\x_k = \x^k_{t_k + 1}$
		\STATE Update  $t_k$
		\ENDFOR
		\STATE \textbf{Output}:  $\x_K$
	\end{algorithmic}
\end{algorithm}

Next, we show that APG can be made adaptive to HEB by periodically restarting given $c$ and $\theta$. This is similar to~\citep{DBLP:journals/corr/nesterov16linearnon} under the QGC. 
The steps of restarting APG (rAPG) are presented in Algorithm~\ref{alg:2}, where we employ the simplest variant of APG. 
\begin{thm}\label{thm:rAPG}
	Suppose $F(\x_0) - F_*\leq \epsilon_0$ and $F(\x)$ satisfies HEB on $\S_{\epsilon_0}$.  By running Algorithm~\ref{alg:1} with $K = \lceil \log_2\frac{\epsilon_0}{\epsilon}\rceil$ and $t_k=\lceil 2c\sqrt{L}\epsilon_{k-1}^{\theta -1/2} \rceil$, we have $F(\x_K) - F_* \leq \epsilon$. 
	The iteration complexity of rAPG is $O(c\sqrt{L}\epsilon_0^{1/2-\theta})$ if $\theta>1/2$, and if $\theta\leq 1/2$ it is $O(\max\{\frac{c\sqrt{L}}{\epsilon^{1/2 - \theta}}, c\sqrt{L}\log(\frac{\epsilon_0}{\epsilon})\})$. 
\end{thm}
\begin{proof}
	Similar to the proof of Theorem \ref{thm:1}, we will prove by induction that $F(\x_k) - F_*\leq \epsilon_k \triangleq \frac{\epsilon_0}{2^k}$. Assume that $F(\x_{k-1}) - F_*\leq \epsilon_{k-1}$. Hence,  $\x_{k-1}\in\S_{\epsilon_0}$. Then according to Proposition \ref{prop:apg} and the HEB condition, we have
	\begin{align*}
	F(\x_k ) - F_* \leq  \frac{2c^2L\epsilon_{k-1}^{2\theta}}{ (t_k + 1)^2}.
	\end{align*}
	Since $t_k\geq   2 c\sqrt{L}\epsilon_{k-1}^{\theta -1/2}$,  we have
	\begin{align*}
	F(\x_k) - F_* \leq \frac{\epsilon_{k-1}}{2}= \epsilon_k.
	\end{align*}
	After $K$ stages, we have $F(\x_K) - F_*\leq \epsilon$. 
	The total number of iterations is 
	\begin{align*}
	T_K=\sum_{k=1}^Kt_k  \leq  O(c\sqrt{L}\epsilon_{k-1}^{\theta -1/2}).
	\end{align*}
	When $\theta> 1/2$, we have $T_K \leq O(c\sqrt{L}\epsilon_0^{\theta - 1/2})$. When $\theta\leq 1/2$, we have
	\[
	T_K \leq O\left(\max\{c\sqrt{L}\log(\epsilon_0/\epsilon), c\sqrt{L}/\epsilon^{1/2-\theta}\}\right).
	\]
\end{proof}


From Algorithm~\ref{alg:2}, we can see that rAPG requires the knowledge of $c$ besides $\theta$ to restart APG. However, for many problems of interest, the value of $c$ is unknown, which makes rAPG impractical. To address this issue, we propose to use the magnitude of the proximal gradient as a measure for restart and termination. Previous work~\citep{opac-b1104789} have considered the strongly convex optimization problems where the strong convexity parameter is unknown, where they also use the magnitude of the proximal gradient as a measure for restart and termination. However, in order to achieve faster convergence under the HEB condition without the strong convexity, we have to introduce a novel technique of adaptive regularization that adapts to the HEB. With a novel synthesis of the adaptive regularization and a conditional restarting that searchs for the $c$, we are able  to develop practical adaptive  accelerated gradient methods.

Before diving into the details of the proposed algorithm, we will first present a variant of PG as a baseline for comparison motivated by~\citep{nesterov8812} for smooth problems, which enjoys a faster convergence than the vanilla PG in terms of the proximal gradient's norm. The idea is to return a solution that achieves the minimum magnitude of the proximal gradient (the option II in Algorithm~\ref{alg:1}). 
The convergence of $\min_{1\leq \tau\leq t}\|G(\x_\tau)\|_2$ under HEB is presented in the following theorem.

\begin{thm}\label{thm:3}
	Suppose $F(\x_0) - F_*\leq \epsilon_0$ and $F(\x)$ satisfies HEB on $\S_{\epsilon_0}$. The iteration complexity of PG (with option II) for achiving  $\min_{1\leq \tau\leq t}\|G(\x_\tau)\|_2 \leq \epsilon$,  is $O(c^{\frac{1}{1-\theta}}L\max\{1/\epsilon^{\frac{1 - 2\theta}{1-\theta}}, \log(\frac{\epsilon_0}{\epsilon})\})$ if $\theta\leq 1/2$, and is $O(c^2L\epsilon_0^{2\theta -1})$ if $\theta>1/2$. 
\end{thm}

\begin{proof}
	By the update of Algorithm \ref{alg:1} with option II and Corollary \ref{cor:1}, we have
	 \begin{align*}
	F(\x_\tau) - F(\x_{\tau+1})\geq \frac{1}{2L}\|G(\x_\tau)\|_2^2.
	\end{align*}
	Let $t=2j$. Summing over $\tau =j,\ldots, t $ gives
	\begin{align*}
	F(\x_{j}) - F(\x_{t+1})\geq \frac{1}{2L}\sum_{\tau={j}}^{t}\|G(\x_\tau)\|_2^2.
	\end{align*}
	Since $\|G(\x_\tau)\|_2\geq \min_{1\leq \tau\leq t}\|G(\x_\tau)\|_2$ and $F(\x_{t+1})\geq F_*$,  we have 
	$$ \frac{j}{2L}\min_{1\leq \tau\leq t}\|G(\x_\tau)\|_2^2\leq F(\x_{j}) - F_*.$$
	Hence, 
	\begin{align}
	\label{proofThm3}
	\min_{1\leq \tau\leq t}\|G(\x_\tau)\|_2^2\leq\frac{2L}{j}(F(\x_{j}) - F_*).
	\end{align}
	We consider three scenarios of $\theta$. \\

	(I). If $\theta>1/2$, according to Theorem \ref{thm:1}, we know that $F(\x_j)-F_*$ converges to $0$ in $j=O(c^2L\epsilon_0^{2\theta-1})$ steps, so $\min_{1\leq \tau\leq t}\|G(\x_\tau)\|_2^2$ converges to $0$ in $t=O(c^2L\epsilon_0^{2\theta-1})$ steps.\\
	
	(II). If $\theta=1/2$, 
	let $j=\max(k,2L)$ and $t=2j$,
	where $k=ac^2L\log{(\frac{\epsilon_0}{\epsilon^2})}$, and $a$ is a constant hided in the big O notation. According to Theorem 1, we have
	\begin{equation}
	\label{proofThm3_2}
	F(\x_k)-F_*\leq\epsilon^2,
	\end{equation}
	then the inequality~(\ref{proofThm3}), (\ref{proofThm3_2}) and the choice of $j,k$ yields
	\begin{align*}
	\min_{1\leq \tau\leq t}\|G(\x_\tau)\|_2^2\leq\frac{2L}{j}(F(\x_{j}) - F_*)\leq\epsilon^2,
	\end{align*}
	so we know that $t=O(c^2L\log{(\frac{\epsilon_0}{\epsilon}}))$.\\
	
	(III). If $\theta<1/2$,  let $j$ be an index such that $F(\x_j) - F_*\leq \epsilon'$. We can set  $j = 2ac^2L/{\epsilon'}^{1-2\theta}$ and hence $t =  4ac^2L/{\epsilon'}^{1-2\theta}$, and have
	\begin{align*}
	\min_{1\leq \tau\leq t}\|G(\x_\tau)\|_2^2\leq\frac{2L}{j}(F(\x_{j}) - F_*)\leq \frac{\epsilon' {\epsilon'}^{1-2\theta}}{ac^2} = \frac{{\epsilon'}^{2-2\theta}}{ac^2}.
	\end{align*}
	Let $\epsilon' = c^{\frac{1}{1-\theta}}\epsilon^{\frac{1}{(1-\theta)}}$, we have $\min_{1\leq \tau\leq t}\|G(\x_\tau)\|^2_2\leq \epsilon^2/a$. We can conclude $t = O(c^{\frac{1}{1-\theta}}L/\epsilon^{\frac{1-2\theta}{1-\theta}})$.\\

	By combining the three scenarios, we can complete the proof. 
\end{proof}

The final theorem  in this section summarizes an $o(1/t)$ convergence result of PG  for minimizing a proper, coercive, convex, lower semicontinuous and semi-algebraic function, which could be interesting of its own. 
\begin{thm}\label{thm:pgso}
	Let $F(\x)$ be a proper, coercive, convex, lower semicontinuous and semi-algebraic functions. Then PG  (with option I and option II) converges at a speed of $o(1/t)$ for $F(\x) - F_*$ and $G(\x)$, respectively, where $t$ is the total number of iterations. 
\end{thm}
{\bf Remark:} This can be easily proved by combining Proposition~\ref{prop:5} and Theorems~\ref{thm:1},~\ref{thm:3}.

\section{Adaptive Accelerated Gradient Converging Methods for Smooth Optimization}
In the following two sections, we will present adaptive accelerated gradient converging methods that are faster than minPG for the convergence of (proximal) gradient's norm. Due to its simplicity, we first consider the following unconstrained optimization problem: 
\begin{align*}
\min_{\x\in\R^d}f(\x)
\end{align*}
where $f(\x)$ is a $L$-smooth function. We abuse $\Omega_*$ to denote the optimal set of above problem.  
The lemma below that bounds the distance of a point to the optimal set by a function of the gradient's norm.  
\begin{lemma}\label{lem:4}
	If $f(\x)$ satisfies the HEB on $\x\in\S_{\xi}$ with $\theta\in(0,1]$, i.e., there exists $c>0$ such that for any $\x\in\S_{\xi}$ we have
	\[
	D(\x, \Omega_*)\leq c(f(\x) - f_*)^\theta.
	\]
	If $\theta\in(0,1)$, then for any $\x\in\S_\xi$
	\begin{align*}
	D(\x, \Omega_*)\leq c^{\frac{1}{1-\theta}}\|\partial  f(\x)\|_2^{\frac{\theta}{1-\theta}}.
	\end{align*}
	If $\theta=1$, then for any $\x\in\S_\xi$
	\begin{align*}
	D(\x, \Omega_*)\leq c^2\xi\|\partial  f(\x)\|_2.
	\end{align*}
\end{lemma} 
The proof of this lemma is included in the Appendix.

Note that for a smooth function $f(\x)$, we can restrict our discussion on HEB condition to $\theta\leq 1/2$. Since $f(\x) - f(\x_*)\leq \frac{L}{2}\|\x- \x_*\|_2^2$ where $\x_*\in\Omega_*$, plugging this equality into the HEB we can see $\theta$ has to be less than $1/2$ if $c$ remains a constant. 
 In order to derive faster convergence than minPG, we employ the technique of regularization, i.e., adding a strongly convex regularizer into the objective. To this end, we define the following problem: 
\begin{align*}
f_\delta(\x) = f(\x) + \frac{\delta}{2}\|\x - \x_0\|_2^2,
\end{align*}
where $\x_0$ is the initial solution. It is clear that $f_\delta(\x)$ is a $(L+\delta)$-smooth and $\delta$-strongly convex function.  The proposed adaAGC algorithm will run in multiple stages. At the $k$-th stage, we construct a problem like above using a value of $\delta_{k}$ and an initial solution $\x_{k-1}$, and employ APG for smooth and strongly convex minimization to solve the constructed problem until the gradient's norm is decreased by a factor of $2$. The initial solution for each stage is the output solution of the previous  stage and the value of $\delta$ will be adaptively decreasing based on $\theta$ in the HEB condition. Specifically, the choice of $\delta_k$ can be set in the following way:
\begin{equation}
\label{delta_smooth}
	\delta_k =
	 \frac{\epsilon_{k-1}^{\frac{1-2\theta}{1-\theta}}}{6c_e^{1/(1-\theta)}}
\end{equation}
 We also embed a search procedure for the value of $c$ into the algorithm in order to leverage the HEB condition. The detailed steps of adaAGC for solving $\min_{\x}f(\x)$ are presented in Algorithm~\ref{alg:1-4} assuming $f(\x)$ satisfies a HEB condition. 

\begin{algorithm}[t]
\caption{adaAGC for solving~(\ref{eqn:smt}) } \label{alg:1-4}
\begin{algorithmic}[1]
\STATE \textbf{Input}: $\x_0\in\Omega$ and $c_0$ and $\gamma>1$
\STATE Let $c_e=c_0$ and $\epsilon_0 = \|\nabla f(\x_0)\|_2$, 
\FOR{$k=1,\ldots, K$}
   \STATE Let $\delta_k$ be given in (\ref{delta_smooth}) and $f_{\delta_k}(\x)  =f(\x) + \frac{\delta_k}{2}\|\x - \x_{k-1}\|_2^2$
   \STATE $\x^k_{1} = \x_{k-1}$ and $\y_1 = \x_{k-1}$
  \FOR{$s=1, \ldots$}
    \FOR{$\tau = 1, \ldots$}
            \STATE $\x^k_{\tau+1} = \y_\tau - \frac{1}{L+\delta_k}\nabla f_{\delta_k}(\y_\tau)$
        \STATE $\y_{\tau+1} = \x^k_{\tau+1} + \frac{\sqrt{L+\delta_k} - \sqrt{\delta_k}}{\sqrt{L + \delta_k} + \sqrt{\delta_k}} (\x^k_{\tau+1} - \x^k_{\tau})$
        \IF{$\|\nabla  f(\x^k_{\tau+1})\|_2\leq \epsilon_{k-1}/2$ }
        \STATE let $\x_k =\x^k_{\tau+1}$ and $\epsilon_k = \epsilon_{k-1}/2$.
        \STATE break the two enclosing for loops 
        \ELSIF{$\tau=\left\lceil  2\sqrt{\frac{L+\delta_k}{\delta_k}}\log\frac{\sqrt{L(L+\delta_k)}}{\delta_k}\right\rceil$}
        \STATE let $c_e = \gamma c_e$ and break the enclosing for loop
        \ENDIF
    
    \ENDFOR
    \ENDFOR
\ENDFOR
\STATE \textbf{Output}:  $\x_K$
\end{algorithmic}
\end{algorithm}

Below, we first present the analysis for each stage to pave the path of proof for our main theorem. 
\begin{thm}\label{thm:on}
Suppose $f(\x)$ is $L$-smooth. By running the update in~(\ref{eqn:APG}) for solving $f_\delta(\x) = f(\x) + \frac{\delta}{2}\|\x - \x_0\|_2^2$ with $\beta = \frac{\sqrt{L+\delta} - \sqrt{\delta}}{\sqrt{L+\delta} + \sqrt{\delta}}$ and an initial solution $\x_0$, we have for any $\x\in\R^d$
\begin{align*}
f_\delta(\x_{t+1}) - f_\delta(\x)\leq \left(1 - \sqrt{\frac{\delta}{L+\delta}}\right)^t\left[f(\x_0) - f(\x) \right],
\end{align*}
and $f(\x_{t+1}) \leq f(\x_0)$. 
If $t\geq \sqrt{\frac{L+\delta}{\delta}}\log\left(\frac{L}{\delta}\right)$, we have
\begin{align*}
\|\x_{t+1} - \x_0\|_2 \leq \sqrt{2}\|\x_0- \x_*\|_2.
\end{align*}
\end{thm}
\begin{proof}
By Proposition~\ref{prop:3}, we have 
\[
f_\delta(\x_{t+1}) - f_\delta(\x)\leq \left(1 - \sqrt{\frac{\delta}{L+\delta}}\right)^t\left[f_\delta(\x_0) - f_\delta(\x) + \frac{\delta}{2}\|\x - \x_0\|_2^2 \right].
\]
By noting the definition of $f_\delta(\x)$ we can prove the first inequality. To prove the second inequality, we let $\x=\x_0$ in the first inequality, we have
\begin{align*}
f(\x_{t+1}) + \frac{\delta}{2}\|\x_{t+1} - \x_0\|_2^2 - f(\x_0)\leq 0.
\end{align*}
Thus $f(\x_{t+1})\leq f(\x_0)$. To prove the third inequality, we let $\x=\x_*\in\Omega_*$ in the first inequality, we have
\begin{align*}
f(\x_{t+1}) + \frac{\delta}{2}\|\x_{t+1} - \x_0\|_2^2 - f(\x_*) -  \frac{\delta}{2}\|\x_0 - \x_*\|_2^2\leq \left(1 - \sqrt{\frac{\delta}{L+\delta}}\right)^t\left[f(\x_0) - f(\x_*)  \right].
\end{align*}
Then we have
\begin{align*}
\|\x_{t+1} - \x_0\|_2^2\leq  \|\x_0 - \x_*\|_2^2+\left(1 - \sqrt{\frac{\delta}{L+\delta}}\right)^t\frac{L}{\delta}\|\x_0 - \x_*\|_2^2.
\end{align*}
If $t\geq \sqrt{\frac{L+\delta}{\delta}}\log\left(\frac{L}{\delta}\right)$, we have
\begin{align*}
\|\x_{t+1} - \x_0\|_2^2&\leq  \|\x_0 - \x_*\|_2^2+\left(1 - \sqrt{\frac{\delta}{L+\delta}}\right)^t\frac{L}{\delta}\|\x_0 - \x_*\|_2^2\\
&\leq \|\x_0 - \x_*\|_2^2 + \|\x_0 - \x_*\|_2^2 = 2\|\x_0 - \x_*\|_2^2.
\end{align*}
\end{proof}
Next, we prove the following theorem. 
\begin{thm}\label{thm:ong}
Under the same condition as in Theorem~\ref{thm:on}, we have 
\[
\|\nabla f(\x_{t+1})\|_2\leq\sqrt{L(L+\delta)} \left(1 - \sqrt{\frac{\delta}{L+\delta}}\right)^{t/2}\|\x_0- \x_*\|_2 + \sqrt{2}\delta\|\x_0 - \x_*\|_2.
\]
\end{thm}
\begin{proof}
Let $\x_{t+1}^+ =\x_{t+1} - \frac{1}{L+\delta}\nabla f_\delta(\x_{t+1})$ in the first inequality in Theorem~\ref{thm:on}, we have
\begin{align*}
f_\delta(\x_{t+1}) - f_\delta(\x_{t+1}^+) &\leq \left(1 - \sqrt{\frac{\delta}{L+\delta}}\right)^{t}[f(\x_0) - f(\x_{t+1}^+)]\leq \left(1 - \sqrt{\frac{\delta}{L+\delta}}\right)^{t}[f(\x_0) - f(\x_*)],
\end{align*}
where the last inequality uses  $f(\x_{t+1}^+)\geq f(\x_*)$.
Applying Corollary~\ref{cor:1} we have
\begin{align*}
\frac{1}{2(L+\delta)}\|\nabla f_\delta(\x_{t+1})\|_2^2\leq f_\delta(\x_{t+1}) - f_\delta(\x_{t+1}^+)\leq  \left(1 - \sqrt{\frac{\delta}{L+\delta}}\right)^{t}[f(\x_0) - f(\x_*)].
\end{align*}
Then we have
\begin{align*}
\|\nabla f_\delta(\x_{t+1})\|_2\leq \sqrt{2(L+\delta)} &\left(1 - \sqrt{\frac{\delta}{L+\delta}}\right)^{t/2}\sqrt{f(\x_0) - f(\x_*)}\\
&\leq  \sqrt{L(L+\delta)} \left(1 - \sqrt{\frac{\delta}{L+\delta}}\right)^{t/2}\|\x_0 - \x_*\|_2,
\end{align*}
where the last inequality uses the smoothness of $f(\x)$. 
To proceed, we have
\begin{align*}
\|\nabla f(\x_{t+1})\|_2& = \|\nabla f_\delta(\x_{t+1}) - \delta(\x_{t+1} - \x_0)\|_2 \leq \|\nabla f_\delta(\x_{t+1})\|_2 + \delta\|\x_{t+1} - \x_0\|_2\\
&\leq  \|\nabla f_\delta(\x_{t+1})\|_2 + \sqrt{2}\delta\|\x_0 - \x_*\|_2.
\end{align*}
\end{proof}
Finally, we can prove the main theorem of this section. 
\begin{thm}
	\label{smooth_thm}
Suppose $f(\x_0) - f_*\leq \xi_0$, $f(\x)$ satisfies HEB on $\S_{\xi_0}$ with $\theta\in(0,1]$ and $c_0\leq c$. Let $\epsilon_0 = \|\nabla f(\x_0)\|_2$ and $K=\lceil \log_2(\frac{\epsilon_0}{\epsilon})\rceil$, $p=(1-2\theta)/(1-\theta)$ for $\theta\in(0,1/2]$. The iteration complexity of the Algorithm~\ref{alg:1-4} for having $\|\nabla f(\x_K)\|_2 \leq \epsilon$ is $ \textstyle \widetilde O\left(\sqrt{L}c^{\frac{1}{2(1-\theta)}}\max(\frac{1}{\epsilon^{p/2}}, \log(\varepsilon_0/\epsilon) \right)$, 
where $\widetilde O(\cdot)$ suppresses a log term depending on  $c, c_0, L, \gamma$.
\end{thm}
\begin{proof}

We can easily induce that $f(\x_{k}) - f_*\leq \xi_0$ from Theorem~\ref{thm:on}. Let $t_k =  \lceil 2\sqrt{\frac{L+\delta_k}{\delta_k}}\log\frac{\sqrt{L(L+\delta_k)}}{\delta_k}\rceil$.
 Applying Theorem~\ref{thm:ong} to the $k$-the stage of adaAGC, we have
\begin{align*}
&\|\nabla f(\x^k_{t_k+1})\|_2\leq\sqrt{L(L+\delta_k)} \left(1 - \sqrt{\frac{\delta_k}{L+\delta_k}}\right)^{t_k/2}\|\x_{k-1}  - \x_*\|_2+ \sqrt{2}\delta_k\|\x_{k-1} - \x_*\|_2\\ 
&\leq \sqrt{L(L+\delta_k)} \left(1 - \sqrt{\frac{\delta_k}{L+\delta_k}}\right)^{t_k/2}c^{\frac{1}{(1-\theta)}}\|\nabla f(\x_{k-1})\|_2^{\frac{\theta}{(1-\theta)}} + \sqrt{2}\delta_k c^{\frac{1}{(1-\theta)}}\|\nabla f(\x_{k-1})\|_2^{\frac{\theta}{(1-\theta)}},
\end{align*}
where the last inequality follows Lemma~\ref{lem:4}.  Note that at each stage,  we check two conditions (i) $\|\nabla f(\x^k_{\tau+1})\|_2\leq \epsilon_{k-1}/2$ and  (ii) $\tau=t_k$. If the first condition satisfies first, we proceed to the next stage. If the second condition satisfies first, then we can claim that $c_e\leq c$ and then we  increase $c_e$ by a factor $\gamma>1$ and then restart the same stage. To verify the claim, assume $c_e>c$ and the second condition satisfies first, i.e., $\tau=t_k$ but $\|\nabla f(\x^k_{\tau+1})\|_2>\epsilon_{k-1}/2$. We will deduce a contradiction. To this end,  we use
\begin{align*}
\|\nabla f(\x^k_{t_k+1})\|_2&\leq\left( \sqrt{L(L+\delta_k)} \left(1 - \sqrt{\frac{\delta_k}{L+\delta_k}}\right)^{t_k/2}+\sqrt{2}\delta_k\right)c^{\frac{1}{1-\theta}}\|\nabla f(\x_{k-1})\|_2^{\frac{\theta}{1-\theta}} \\
&\leq (\delta_k  + \sqrt{2}\delta_k)c^{\frac{1}{1-\theta}}\|\nabla f(\x_{k-1})\|_2^{\frac{\theta}{1-\theta}} \\
&\leq 3 \delta_k c^{\frac{1}{1-\theta}}\|\nabla f(\x_{k-1})\|_2^{\frac{\theta}{1-\theta}} = \frac{3\epsilon_{k-1}^{\frac{1-2\theta}{1-\theta}}}{6c_e^{1/(1-\theta)}} c^{\frac{1}{1-\theta}}\epsilon_{k-1}^{\frac{\theta}{1-\theta}}\leq \epsilon_{k-1}/2,
\end{align*}
where the last inequality follows that $c_e>c$. This contradicts to the assumption that $\|\nabla f(\x^k_{\tau+1})\|_2>\epsilon_{k-1}/2$, which verifies our claim.

Since $c_e$ is increased by a factor $\gamma>1$ whenever condition (ii) holds first.  Thus with at most $\lceil \log_\gamma(c/c_0)\rceil$ times condition (ii) holds first. Similarly with at most $\lceil\log_2\epsilon_0/\epsilon\rceil$ times that condition (i) holds first before the algorithm terminates. We let $T_k$ denote the total number of iterations in order to make condition (i) satisfies in stage $k$. First, we can see that $c_e\leq \gamma c$. 

Let  $\delta'_k = \frac{\epsilon_{k-1}^{\frac{1-2\theta}{1-\theta}}}{6(\gamma c)^{1/(1-\theta)}}$ and $t'_k =   \lceil 2\sqrt{\frac{L+\delta'_k}{\delta'_k}}\log\frac{\sqrt{L(L+\delta'_k)}}{\delta'_k}\rceil$. Let $s_k$ denote the number of cycles in each stage in order to have $\|\nabla f(\x^k_{\tau+1})\|_2\leq \epsilon_k$. Then $s_k\leq \log_\gamma(c/c_0)+1$.  The total number of iterations of across all stages is bounded by   $\sum_{k=1}^Ks_kt_k$, which is bounded by 
\begin{align*}
\sum_{k=1}^Ks_kt_k &\leq (1+\log_{\gamma}(c/c_0))\sum_{k=1}^Kt_k'.
\end{align*}
Plugging the value of $t_k'$, we can deduce the iteration complexity in Theorem \ref{smooth_thm} for $\theta\in(0,1/2]$.

\end{proof}

\section{Adaptive Accelerated Gradient Converging Methods  for Smooth
Composite Optimization}

In this section, we generalize the results in previous section to smooth composite optimization problem
\begin{align*}
\min_{\x\in\R^d}F(\x)\triangleq f(\x) + g(\x).
\end{align*}
Different from last section, we will use the proximal gradient $G(\x_t)$ as a measure for restart and termination in adaAGC. 
Similar to last section, we first  present a key lemma for our development that serves the foundation of the adaptive regularization and conditional restarting. 

\begin{lemma}\label{lem:5}
	Assume $F(\x)$ satisfies HEB for any $\x\in\S_{\xi}$ with $\theta\in(0,1]$. If $\theta\in(0,1/2]$ then we have for any $\x\in\S_{\xi}$
	\begin{align*}
		D(\x, \Omega_*)\leq \frac{2}{L}\|G(\x)\|_2 + c^{\frac{1}{1-\theta}}2^{\frac{\theta}{1-\theta}}\|G(\x)\|_2^{\frac{\theta}{1-\theta}}.
	\end{align*}
	If $\theta\in(1/2, 1]$, we have
	for any $\x\in\S_{\xi}$
	\begin{align*}
		D(\x, \Omega_*)\leq \left(\frac{2}{L}+ 2c^2\xi^{2\theta-1}\right)\|G(\x)\|_2.
	\end{align*}
\end{lemma}
\begin{proof}
	The conclusion is trivial when $\x\in\Omega_*$, so we only need to consider the case when $\x\notin\Omega_*$. Define $P_{\eta F}(\x)=\arg\min\limits_{\u}\frac{1}{2}\Vert \u-\x\Vert_2^2+\eta F(\u)$.
	
	We first prove for $\theta\in(0,1/2]$. It is not difficult to see that $\frac{1}{\eta}(\x - P_{\eta F}(\x))\in\partial F(P_{\eta F}(\x))$. 
	\begin{align*}
	&D(\x, \Omega_*)\leq \|\x - P_{\eta F}(\x)\|_2 + D(P_{\eta F}(\x), \Omega_*)\\
	&\leq \|\x - P_{\eta F}(\x)\|_2 + c^{\frac{1}{1-\theta}}\|\partial F(P_{\eta F}(\x))\|_2^{\frac{\theta}{1-\theta}}\\
	& \leq  \|\x - P_{\eta F}(\x)\|_2 + \frac{c^{\frac{1}{1-\theta}}}{\eta^{\frac{\theta}{1-\theta}}}\|\x - P_{\eta F}(\x)\|_2^{\frac{\theta}{1-\theta}}\\
	& \leq \eta (1+L\eta)\|G_\eta(\x)\|_2 + c^{\frac{1}{1-\theta}}(1+\eta L)^{\frac{\theta}{1-\theta}}\|G_\eta(\x)\|_2^{\frac{\theta}{1-\theta}},
	\end{align*}
	where the second inequality uses the result in Lemma~\ref{lem:4} 
	and the last inequality follows Proposition~\ref{prop_lewis}, which asserts that $\|\x - P_{\eta F}(\x)\|_2\leq \eta(1+L\eta)\|G_\eta(\x)\|_2$. 
	Plugging the value $\eta = 1/L$, we have the result.


	Next, we prove for $\theta\in(1/2,1]$. For any $\x\in S_\xi$, we have $P_{\eta F}(\x)\in S_\xi$ and 
	\begin{align*}
	&D(P_{\eta F}(\x),\Omega_*)\leq c(F(P_{\eta F}(\x))-F_*)^\theta\\
	&=c(F(P_{\eta F}(\x))-F_*)^{1-\theta}(F(P_{\eta F}(\x))-F_*)^{2\theta-1}\leq c^2\Vert \partial F(P_{\eta F}(\x))\Vert_2(F(\x)-F_*)^{2\theta-1}\\
	&\leq c^2\Vert \partial F(P_{\eta F}(\x))\Vert_2 \xi^{2\theta-1}\leq c^2 (1+L\eta)\Vert G_\eta(\x)\Vert_2\xi^{2\theta-1}\leq 2c^2\xi^{2\theta-1}\Vert G_\eta(\x)\Vert_2,
	\end{align*}
	where the second inequality holds because the inequality~(\ref{KL}) holds for any $\theta\in(0,1]$ (by Lemma \ref{lem:4}), $F(P_{\eta F}(\x))\leq F(\x)\leq\xi$, the fourth inequality holds since $\|G_\eta(\x)\|_2\geq \frac{1}{1+L\eta}\left\|(\x - P_{\eta F}(\x))/\eta\right\|_2\geq \frac{1}{1+L\eta}\|\partial F(P_{\eta F}(\x))\|_2$ (by Proposition \ref{prop_lewis}), and the last inequality holds by taking $\eta=1/L$.

	So for $\theta\in(1/2,1]$ and $\eta=1/L$, we have
	\begin{align*}
	D(\x,\Omega_*)&\leq \|\x - P_{\eta F}(\x)\|_2 + D(P_{\eta F}(\x), \Omega_*)\\
	&\leq(\frac{2}{L}+2c^2\xi^{2\theta-1})\Vert G(\x)\Vert_2.
	\end{align*}
\end{proof}

A building block of the proposed algorithm is to solve a problem of the following style:
\begin{align}\label{eqn:delta}
	F_\delta(\x) = F(\x) + \frac{\delta}{2}\|\x - \x_0\|_2^2,
\end{align}
which consists of a $L$-smooth function $f(\x)$ and a $\delta$-strongly convex function $g_\delta(\x) = g(\x) +  \frac{\delta}{2}\|\x - \x_0\|_2^2$. We present some technical results for employing  the Algorithm~\ref{alg:0}  (i.e., Nesterov's ADG) to solve the above problem. 
\begin{thm}\label{thm:on2}
	By running the Algorithm~\ref{alg:0} for minimizing $f(\x) + g_\delta(\x)$ with an initial solution $\x_0$, then for any $\x\in\R^d$ and $t\geq 0$, 
	\begin{align*}
		F_\delta(\x_{t+1}) - F_\delta(\x) \leq \frac{L}{2}\|\x_0 - \x\|_2^2\left[1 + \sqrt{\frac{\delta}{2L}}\right]^{-2t},
	\end{align*}
	and $F(\x_{t+1}) \leq F(\x_0)$. 
	If $t\geq \sqrt{\frac{L}{2\delta}}\log\left(\frac{L}{\delta}\right)$, we have $\|\x_{t+1} - \x_0\|_2 \leq \sqrt{2}\|\x_0- \x_*\|_2$.
\end{thm}
\begin{proof}
	 Applying Proposition \ref{Nesterov2007} to $F_\delta(\x)$ yields
	\begin{equation}
	\label{thm5proof}
	\begin{aligned}
	F(\x_{t+1}) - F(\x) &+ \frac{\delta}{2}\|\x_{t+1} - \x_0\|_2^2 \leq \frac{\delta}{2}\|\x - \x_0\|_2^2 \\
	&+   \frac{L}{2}\|\x_0 - \x\|_2^2\left[1 + \sqrt{\frac{\delta}{2L}}\right]^{-2t}.
	\end{aligned}
	\end{equation}
	Then  $
	F(\x_{t+1}) - F(\x_0)\leq 0$,
	and 
	choose $\x=\x_*$ in the inequality (\ref{thm5proof}), where $\x_*\in\Omega_{*}$, then we have
	\begin{align*}
	&\|\x_{t+1} -\x_0\|_2^2\\
	&\leq \|\x_0 - \x_*\|_2^2 + \frac{L}{\delta}\|\x_0 - \x_*\|_2^2\left[1 + \sqrt{\frac{\delta}{2L}}\right]^{-2t}.
	\end{align*}
	Under the condition $t\geq \sqrt{\frac{L}{2\delta}}\log\left(\frac{L}{\delta}\right)$ we have $\|\x_{t+1} - \x_0\|_2\leq \sqrt{2}\|\x_0 - \x_*\|_2$.
\end{proof}
 

\begin{thm}\label{thm:ong2}
	Under the same condition as in Theorem~\ref{thm:on2}, for $t\geq \sqrt{\frac{L}{2\delta}}\log\left(\frac{L}{\delta}\right)$ we have
\begin{align*}
		\|G(\x_{t+1})\|_2  &\leq \sqrt{L(L+\delta)}\|\x_0 - \x_*\|_2\left[1 + \sqrt{\delta/(2L)}\right]^{-t}\\
		& + 2\sqrt{2}\delta\|\x_* - \x_0\|_2.
	\end{align*}
\end{thm}
\begin{proof}
	Let $\x^*_\delta$ be the optimal solution to $\min_{\x\in\R^d}F_\delta(\x)$ and $\x_*$ denote an optimal solution to $\min_{\x\in\R^d}F(\x)$.  Thanks to the strong convexity of $F_\delta(\x)$, we  have
	$F_\delta(\x_*) - F_\delta(\x^*_\delta)\geq \frac{\delta}{2}\|\x_* - \x^*_\delta\|_2^2$.
	Then 
	\begin{align*}
		F(\x_*) - F(\x_\delta^*) &+ \delta/2\|\x_* - \x_0\|_2^2  - \delta/2\|\x_\delta^* - \x_0\|_2^2\geq  \delta/2\|\x_* - \x^*_\delta\|_2^2.
	\end{align*}
	Since $F(\x_*) - F(\x^*_\delta)\leq 0$, it implies  $\|\x^*_\delta - \x_0\|_2 \leq \|\x_* - \x_0\|_2$.   By Corollary~\ref{cor:1}, we have
	\begin{align*}
		\frac{\eta}{2}\|G_\eta^\delta(\x_{t+1})\|_2^2& \leq F_\delta(\x_{t+1}) - F_\delta(\x^*_{\delta})\leq  \frac{L}{2}\|\x_0 - \x_\delta^*\|_2^2\left[1 + \sqrt{\delta/(2L)}\right]^{-2t},
	\end{align*}
	where $\eta\leq 1/(L+\delta)$ and $G_\eta^\delta$ is a proximal gradient of $F_\delta(\x)$  defined as $ G^\delta_\eta(\x) =\frac{1}{\eta}\left(\x - \x^+_\eta(\delta)\right)$ and 
	\begin{align*}
		\x_\eta^+(\delta) = \arg\min_{\y}\left\{\eta (\nabla f(\x)+\delta (\x - \x_0))^{\top}(\y - \x)+\eta g(\y)+ \frac{1}{2}\|\y - \x\|_2^2\right\}.
	\end{align*}
	Recall that $\x_\eta^+=P_{\eta g}(\x - \eta\nabla f(\x))$. 
	It is not difficult to derive that $\|\x_\eta^+ - \x^+_\eta(\delta)\|_2\leq 2\eta\delta\|\x - \x_0\|_2$ (by Lemma \ref{strongly} in the appendix).  Since $G_\eta(\x) =\frac{1}{\eta}(\x - \x_\eta^+)$, we have
	\begin{align*}
		\|G_\eta(\x)\|_2 & \leq \|G_\eta^\delta(\x)\|_2  + \|\x_\eta^+ - \x_\eta^+(\delta)\|_2/\eta\leq\|G_\eta^\delta(\x)\|_2  +2\delta\|\x - \x_0\|_2.
	\end{align*}
	Let $\eta = 1/(L+\delta)$, we have
	\begin{align*}
		&\|G_\eta(\x_{t+1})\|_2\leq 2\delta\|\x_{t+1} - \x_0\|_2+\sqrt{L/\eta}\|\x_0 - \x_{\delta}^*\|_2\left[1 + \sqrt{\delta/(2L)}\right]^{-t}\\
		&\leq 2\sqrt{2}\delta\|\x_* - \x_0\|_2+\sqrt{L(L+\delta)}\|\x_0 - \x_*\|_2\left[1 + \sqrt{\delta/(2L)}\right]^{-t}.
	\end{align*}
	where we use the inequality  $\|\x_\delta^*-\x_0\|_2\leq \|\x_* - \x_0\|_2$.  Since $\|G_\eta(\x)\|_2$ is a monotonically decreasing function of $\eta$ (by the Proposition \ref{decreaseNesterov})
	, then $\|G(\x)\|_2 \leq\|G_\eta(\x)\|_2$ for $\eta= 1/(L+\delta)\leq 1/L$. Then
\begin{align*}
		\|G(\x_{t+1})\|_2  &\leq \sqrt{L(L+\delta)}\|\x_0 - \x_*\|_2\left[1 + \sqrt{\delta/(2L)}\right]^{-t} + 2\sqrt{2}\delta\|\x_* - \x_0\|_2.
	\end{align*}
\end{proof}
\begin{algorithm}[t]
	\caption{adaAGC for solving~(\ref{eqn:opt}) } \label{alg:1-5}
	\begin{algorithmic}[1]
		\STATE \textbf{Input}: $\x_0\in\Omega$ and $c_0$ and $\gamma>1$
		\STATE Let $c_e=c_0$ and $\varepsilon_0 = \|G(\x_0)\|_2$, 
		\FOR{$k=1,\ldots, K$}
		\FOR{$s=1,\ldots, $}
		\STATE Let $\textstyle\delta_k$ be given in~(\ref{eqn:deltak}) and $g_{\delta_k}(\x)  =g(\x)  + \frac{\delta_k}{2}\|\x - \x_{k-1}\|_2^2$
		
		\STATE $A_0=0$, $\v_0 =\x_{k-1}$,   $\x^k_{0} = \x_{k-1}$
		\FOR{$t= 0, \ldots$}
		\STATE Let $a_{t+1}$ be the root of $\frac{a^2}{A_t+a} = 2\frac{1+\delta_k A_t}{L}$
		\STATE Set $A_{t+1} = A_t+ a_{t+1}$
		\STATE Set $\y_t = \frac{A_t}{A_{t+1}}\x_t^k + \frac{a_{t+1}}{A_{t+1}}\v_t$
		\STATE Compute $\x^k_{t+1} =P_{g_{\delta_k}/L}(\y_{t} - \nabla f(\y_t)/L) $
		\STATE Compute $\v_{t+1} = \arg\min_{\x} \frac{1}{2}\|\x - \x_{k-1}\|_2^2+\sum_{\tau=1}^{t+1}a_\tau\nabla f(\x^k_\tau)^{\top}\x + A_{t+1}g_{\delta_k}(\x)$
		\IF{$\|G(\x^k_{t+1})\|_2\leq \varepsilon_{k-1}/2$ }
		\STATE let $\x_k =\x^k_{t+1}$ and $\varepsilon_k = \varepsilon_{k-1}/2$.
		\STATE break the enclosing two  for loops 
		\ELSIF{$\tau=\lceil  \sqrt{\frac{2L}{\delta_k}}\log\frac{\sqrt{L(L+\delta_k)}}{\delta_k}\rceil$}
		\STATE let $c_e = \gamma c_e$ and break the enclosing for loop
		\ENDIF
		\ENDFOR
		\ENDFOR
		
		\ENDFOR
		\STATE \textbf{Output}:  $\x_K$
	\end{algorithmic}
\end{algorithm}

Finally, we present the proposed adaptive accelerated gradient converging (adaAGC) method for solving the smooth composite optimization in Algorithm~\ref{alg:1-5} and  prove the main theorem of this section. The adaAGC runs with multiple stages ($k=1,\ldots, K$). We start with an initial guess $c_0$ of the parameter $c$ in the HEB. With the current guess $c_e$ of $c$, at the $k$-th stage adaAGC employs ADG to solve a problem of~(\ref{eqn:delta}) with an adaptive regularization parameter $\delta_k$ being 
\begin{equation}\label{eqn:deltak}
	\delta_k =\left\{\begin{array}{lc}\min\left(\frac{L}{32}, \frac{\varepsilon_{k-1}^{\frac{1-2\theta}{1-\theta}}}{16c_e^{1/(1-\theta)}2^{\frac{\theta}{1-\theta}}}\right)&\text{ if }\theta\in(0,1/2]\\\min\left(\frac{L}{32}, \frac{1}{32c_e^2\epsilon_0^{2\theta-1}}\right)&\text{ if }\theta\in(1/2,1] \end{array}\right.
\end{equation}
The step 16 specifies the condition for restarting with an increased value of $c_e$. When the flow enters step 17 before step 14 for each $s$, it means that the current guess $c_e$ is not sufficiently large, then we increase $c_e$ and repeat the same process (next iteration for $s$). We refer to this machinery as conditional restarting. 
\begin{thm}\label{thm:main}
	Suppose $F(\x_0) - F_*\leq \epsilon_0$, $F(\x)$ satisfies HEB on $\S_{\epsilon_0}$ and $c_0\leq c$. Let $\varepsilon_0 = \|G(\x_0)\|_2$,  $K=\lceil \log_2(\frac{\varepsilon_0}{\epsilon})\rceil$, $p=(1-2\theta)/(1-\theta)$ for $\theta\in(0,1/2]$. 
	The iteration complexity of Algorithm~\ref{alg:1-5} for having $
	\|G(\x_K)\|_2 \leq \epsilon$ is 
	$ \textstyle \widetilde O\left(\sqrt{L}c^{\frac{1}{2(1-\theta)}}\max(\frac{1}{\epsilon^{p/2}}, \log(\varepsilon_0/\epsilon) \right)$  if $\theta\in(0,1/2]$,  and  $\widetilde O(\sqrt{L}c\epsilon_0^{\theta-1/2})$ if $\theta\in(1/2, 1]$, 
	where $\widetilde O(\cdot)$ suppresses a log term depending on  $c, c_0, L, \gamma$.
\end{thm}

\begin{proof}
	\begin{itemize}
	\item We first prove the case when $\theta\in(0,1/2]$. We can easily induce that $F(\x_{k}) - F_*\leq \epsilon_0$ from Theorem~\ref{thm:on2}. Let $t_k =  \lceil \sqrt{\frac{2L}{\delta_k}}\log\frac{\sqrt{L(L+\delta_k)}}{\delta_k}\rceil$.
	Applying Theorem~\ref{thm:ong2} to the $k$-the stage of adaAGC and using Lemma~\ref{lem:5}, we have
\begin{equation}\label{eqn:keyG}
		\begin{aligned}
			&\|G(\x^k_{t_k+1})\|_2
			\leq (\sqrt{L(L+\delta_k)} \left[1 + \sqrt{\frac{\delta_k}{2L}}\right]^{-t_k}+2\sqrt{2}\delta_k)\\
			&\times(\frac{2}{L}\|G(\x_{k-1})\|_2 + c^{\frac{1}{(1-\theta)}}2^{\frac{\theta}{(1-\theta)}}\|G(\x_{k-1})\|_2^{\frac{\theta}{(1-\theta)}} ),
		\end{aligned}
	\end{equation}
	Note that at each stage,  we check two conditions (i) $\|G(\x^k_{\tau+1})\|_2\leq \varepsilon_{k-1}/2$ and  (ii) $\tau=t_k$. If the first condition satisfies first, we proceed to the next stage ($k$ increases by 1). If the second condition satisfies first, then we can claim that $c_e\leq c$ and then we  increase $c_e$ by a factor $\gamma>1$ and then restart the same stage. To verify the claim, assume $c_e>c$ and the second condition satisfies first, i.e., $\tau=t_k$ but $\|G(\x^k_{\tau+1})\|_2>\varepsilon_{k-1}/2$. We will deduce a contradiction. To this end,  we use~(\ref{eqn:keyG}) and note the value of $t_k$
	\begin{align*}
		&\|G(\x^k_{t_k+1})\|_2\leq \left(\delta_k+2\sqrt{2}\delta_k\right)\left(\frac{2}{L}\|G(\x_{k-1})\|_2 + c^{\frac{1}{(1-\theta)}}2^{\frac{\theta}{(1-\theta)}}\|G(\x_{k-1})\|_2^{\frac{\theta}{(1-\theta)}} \right)\\
		&\leq 4\delta_k\left(\frac{2}{L}\|G(\x_{k-1})\|_2 + c^{\frac{1}{(1-\theta)}}2^{\frac{\theta}{(1-\theta)}}\|G(\x_{k-1})\|_2^{\frac{\theta}{(1-\theta)}}\right)\\
		&\leq \frac{\epsilon_{k-1}}{4} + \frac{c^{\frac{1}{(1-\theta)}}2^{\frac{\theta}{(1-\theta)}}\epsilon_{k-1}}{4c_e^{\frac{1}{(1-\theta)}}2^{\frac{\theta}{(1-\theta)}}}\leq \varepsilon_{k-1}/2 = \varepsilon_k,
	\end{align*}
	where the last inequality follows that $c_e>c$. This contradicts to the assumption that $\|G(\x^k_{\tau+1})\|_2>\varepsilon_{k-1}/2$, which verifies our claim. 
	
	Since $c_e$ is increased by a factor $\gamma>1$ whenever condition (ii) holds first, so within at most $\lceil \log_\gamma(c/c_0)\rceil$ times condition (ii) holds first. Similarly with at most $\lceil\log_2\varepsilon_0/\epsilon\rceil$ times that condition (i) holds first before the algorithm terminates. We let $T_k$ denote the total number of iterations in order to make condition (i) satisfies in stage $k$. First, we can see that $c_e\leq \gamma c$.  Let  $\delta'_k = \min(\frac{L}{32},\frac{\varepsilon_{k-1}^{p}}{16(\gamma c2^\theta)^{1/(1-\theta)}})\leq \delta_k$ and $t'_k =   \lceil \sqrt{\frac{2L}{\delta'_k}}\log\frac{\sqrt{L(L+\delta'_k)}}{\delta'_k}\rceil$. Let $s_k$ denote the number of cycles in each stage in order to have $\|G(\x^k_{\tau+1})\|_2\leq \varepsilon_k$. Then $s_k\leq \log_\gamma(c/c_0)+1$.  The total number of iterations of across all stages is bounded by   $\sum_{k=1}^Ks_kt_k$, which is bounded by 
	\begin{align*}
		&\sum_{k=1}^Ks_kt_k \leq (1+\log_{\gamma}(c/c_0))\sum_{k=1}^Kt_k'
	\end{align*}
	Plugging the value of $t_k'$, we can deduce the iteration complexity in Theorem~\ref{thm:main} for $\theta\in(0,1/2]$. 
	\item We consider the proof when $\theta\in(1/2,1]$. Similar to the proof for $\theta\in(0,1/2]$,
	we can easily induce that $F(\x_{k}) - F_*\leq \epsilon_0$ from Theorem \ref{thm:on2}. Let $t_k =  \lceil \sqrt{\frac{2L}{\delta_k}}\log\frac{\sqrt{L(L+\delta_k)}}{\delta_k}\rceil$.
	Applying Theorem~\ref{thm:ong2} to the $k$-the stage of adaAGC and using Lemma~\ref{lem:5}, we have
	\begin{equation}\label{eqn:keyG}
	\begin{aligned}
	&\|G(\x^k_{t_k+1})\|_2
	\leq \left(\sqrt{L(L+\delta_k)} \left[1 + \sqrt{\frac{\delta_k}{2L}}\right]^{-t_k}+2\sqrt{2}\delta_k\right)\left(\frac{2}{L}+2c^2\xi^{2\theta-1}\right)\Vert G(\x_{k-1})\Vert_2,
	\end{aligned}
	\end{equation}
	Note that at each stage,  we check two conditions (i) $\|G(\x^k_{\tau+1})\|_2\leq \varepsilon_{k-1}/2$ and  (ii) $\tau=t_k$. If the first condition satisfies first, we proceed to the next stage ($k$ increases by 1). If the second condition satisfies first, then we can claim that $c_e\leq c$ and then we  increase $c_e$ by a factor $\gamma>1$ and then restart the same stage. To verify the claim, assume $c_e>c$ and the second condition satisfies first, i.e., $\tau=t_k$ but $\|G(\x^k_{\tau+1})\|_2>\varepsilon_{k-1}/2$. We will deduce a contradiction. To this end,  we use~(\ref{eqn:keyG}) and note the value of $t_k$, we have
	\begin{align*}
	&\Vert G(\x_{t_k+1}^k)\Vert_2\leq 4\delta_k\left(\frac{2}{L}+2c^2\xi^{2\theta-1}\right)\Vert G(\x_{k-1})\Vert_2\\
	&\leq \frac{\epsilon_{k-1}}{4}+\frac{8c^2\xi^{2\theta-1}}{32 c_e^2\epsilon_0^{2\theta-1}}\epsilon_{k-1}\leq \frac{\epsilon_{k-1}}{2}=\epsilon_k,
	\end{align*}
	where the last inequality follows that $c_e>c$ and $\xi\leq\epsilon_0$. This contradicts to the assumption that $\|G(\x^k_{\tau+1})\|_2>\varepsilon_{k-1}/2$, which verifies our claim. 
	
	Since $c_e$ is increased by a factor $\gamma>1$ whenever condition (ii) holds first, so within at most $\lceil \log_\gamma(c/c_0)\rceil$ times condition (ii) holds first. Similarly with at most $\lceil\log_2\varepsilon_0/\epsilon\rceil$ times that condition (i) holds first before the algorithm terminates. We let $T_k$ denote the total number of iterations in order to make condition (i) satisfies in stage $k$. First, we can see that $c_e\leq \gamma c$.  Let  $\delta'_k = \min(\frac{L}{32},\frac{1}{32(\gamma c)^2\epsilon_0^{2\theta-1}})\leq \delta_k$ and $t'_k =   \lceil \sqrt{\frac{2L}{\delta'_k}}\log\frac{\sqrt{L(L+\delta'_k)}}{\delta'_k}\rceil$. Let $s_k$ denote the number of cycles in each stage in order to have $\|G(\x^k_{\tau+1})\|_2\leq \varepsilon_k$. Then $s_k\leq \log_\gamma(c/c_0)+1$.  The total number of iterations of across all stages is bounded by   $\sum_{k=1}^Ks_kt_k$, which is bounded by 
	\begin{align*}
	&\sum_{k=1}^Ks_kt_k \leq (1+\log_{\gamma}(c/c_0))\sum_{k=1}^Kt_k'.
	\end{align*}
	Plugging the value of $t_k'$, we can deduce the iteration complexity in Theorem \ref{thm:main} for $\theta\in(1/2,1]$. 
	
	\end{itemize}
\end{proof}

Before ending this section, we would like to remark that if the smoothness parameter $L$ is unknown, one can also employ the backtracking technique pairing with each update to search for $L$~\citep{RePEc:cor:louvco:2007076}. 

\section{Applications}\label{sec:app}
In this section, we present some applications in machine learning and corollaries of our main theorems. In particular, we consider the regularized problems with a smooth loss: 
\begin{align}\label{eqn:pn}
	\min_{\x\in\R^d}\frac{1}{n}\sum_{i=1}^n\ell(\x^{\top}\a_i, b_i) + \lambda R(\x),
\end{align}
where $(\a_i, b_i), i=1,\ldots, n$ denote a set of training examples, $R(\x)$ could be  the $\ell_1$ norm ($\|\x\|_1$), the $\ell_\infty$ norm ($\|\x\|_\infty$),  or a general form $\|\x\|_p^s$ where $p\geq 1$ and $s\in\mathbb N$, or  a huber norm~\citep{DBLP:conf/pkdd/ZadorozhnyiBMSK16}
where $R(\x) = \sum_{i=1}^dh(x_i)$ and $h(x_i)$ is the huber function
\begin{align}\label{eqn:hub}
	h(x)= \left\{\begin{array}{ll}\delta(|x| - \frac{\delta}{2}) & \text{if } |x|\geq \delta/2\\ \frac{\delta^2}{2}&\text{ otherwise}\end{array}\right..
\end{align}

%

We can also consider a broader family of problems that aim to learn a set of models (e.g., in multi-class, multi-label and multi-task  learning): 
\begin{align}\label{eqn:png}
	\min_{\x_1,\ldots, \x_K\in\R^d}\frac{1}{n}\sum_{i=1}^n\sum_{k=1}^K\ell(\x_k^{\top}\a_i, b_{ik}) + \lambda\sum_{k=1}^K\|\x\|_p,
\end{align}
where  the regularizer $\sum_{k=1}^K\|\x\|_p$ is known as $\ell_{1,p}$ norm. 

Next, we present several results about the HEB condition to cover a broad family of loss functions that enjoy the faster convergence of PG and adaAPC. 
\begin{cor}
	Assume the loss function $\ell(z,b)$ is nonnegative,  convex, smooth and semi-algebraic, the the problems in~(\ref{eqn:png}) and~(\ref{eqn:pn}) with $R(\x) = \|\x\|_p^s$ or the huber norm, where $s\in\mathbb N$ and $p\geq 1$ is a rational number,  satisfy the HEB condition with  $\theta\in(0,1]$ on any sublevel set $S_\xi$ with $\xi> 0$. Hence PG have a global convergence speed of $o(1/t)$.  
\end{cor}
{\bf Remark:} Because of the regularization, the objective function is coercive and proper.  The $\ell_p$ norm with $p$ being a rational number is a semi-algebraic function~\citep{Bolte:2014:PAL:2650160.2650169}. Finite sum of semi-algebraic functions and composition of semi-algebraic functions are also semi-algebraic. Then we can use Proposition~\ref{prop:5} and Theorem~\ref{thm:pgso} to prove the above corollary. 
\begin{cor}\label{cor:f}
	Assume the loss function $\ell(z,b)$ is nonnegative, convex, smooth and piecewise quadratic, then the problems in~(\ref{eqn:pn}) and~(\ref{eqn:png}) with $\ell_1$ norm, $\ell_\infty$ norm, Huber norm and $\ell_{1,\infty}$ norm regularization satisfy the HEB condition with $\theta = 1/2$ on any sublevel set $S_\xi$ with $\xi> 0$. Hence adaAGC has a global linear convergence in terms of the proximal gradient's norm and a square root dependence on the condition number. 
\end{cor}
{\bf Remark:} The  above corollary follows directly from Proposition~\ref{prop:quadratic} and Theorem~\ref{thm:main}.  If the loss function is a logistic loss and the regularizer is a polyhedral function (e.g., $\ell_1$, $\ell_\infty$ and $\ell_{1,\infty}$ norm), we can prove the same result as above using Proposition~\ref{prop:log}.  Examples of  convex, smooth and piecewise convex quadratic  loss functions include: square loss: $\ell(z, b) = (z - b)^2$ for $b\in\R$; squared hinge loss: $\ell(z,b) = \max(0, 1 - b z)^2$ for $b\in\{1, -1\}$; and  huber loss: $\ell(z,b)= h(z-b)$ with $h(x)$ defined in~(\ref{eqn:hub}).

		Finally, it is worth mentioning that the result in Corollary~\ref{cor:f}  is more general and better than many previous work. For example, \citet{DBLP:journals/siamjo/Xiao013} and \citet{DBLP:conf/icml/LinX14} only considered the lasso problem consisting of a square loss and $\ell_1$ norm and derived a linear convergence under the {\it restricted eigen-value condition}. The PG has been shown to have a linear convergence for solving the lasso problem~\citep{arxiv:1510.08234,DBLP:journals/corr/nesterov16linearnon,DBLP:conf/pkdd/KarimiNS16,DBLP:journals/corr/abs/1606.00269,Drusvyatskiy16a}, but it has a linear dependence on the condition number. Many works have considered the structured smooth problem $f(\x) = h(A\x) + g(\x)$, where $h(\cdot)$ is a strongly convex function on any compact set~\citep{DBLP:conf/nips/HouZSL13,DBLP:conf/icml/ZhouZS15,DBLP:journals/corr/So13,Luo:1992a,Luo:1992b,Luo:1993}. Note that this structured family does not cover squared hinge loss and huber loss, and mostly the convergence results in these work are local convergence (i.e., asymptotic convergence) instead of global convergence. 
		
		\section{Experimental Results}\label{section:7}

		We conduct some experiments to demonstrate the effectiveness of adaAGC for solving problems of type~(\ref{eqn:opt}). Specifically, we compare adaAGC and PG with option II for optimizing the squared hinge loss (classification), square loss (regression), huber loss ($\delta=1$) (regression) with $\ell_1$ and $\ell_\infty$ regularization, which are cases of~(\ref{eqn:pn}), and we also consider the $\ell_1$ constrained $\ell_p$ norm regression~(\ref{eqn:L1LP}) with varying $p$. We use four datasets from the LibSVM website~\citep{fan2011libsvm}, which are splice ($n=1000, d=60$), german.numer ($n=1000, d=24$) for classification, and  bodyfat ($n=252, d=14$), cpusmall ($n=8192,d=12$) for regression. For problems covered by~(\ref{eqn:pn}), we fix $\lambda=\frac{1}{n}$, and  the parameter $s$ in~(\ref{eqn:L1LP}) is  set to $s=100$ .

\begin{table}[t]
			\caption{Squared hinge loss with $\ell_1$ norm regularization} 
			\label{exp1}
			\begin{center}
				\begin{tabular}{llllll}
					\hline
					Algorithm & dataset & $\epsilon=10^{-4}$ & $\epsilon=10^{-5}$ & $\epsilon=10^{-6}$ & $\epsilon=10^{-7}$\\ 
					\hline
					PG & splice & 2201 & 2201 & 2201 & 2201\\
					adaAGC & splice & 2123  & 2123  & 2123  & 2123\\
					\midrule
					PG &  german.numer &1014 & 1492 & 1971 & 2450\\
					adaAGC & german.numer & 762 & 1010 & 1338 & 1586\\\hline
				\end{tabular}
			\end{center}
			\caption{Square loss with $\ell_1$ norm regularization}
			\label{exp2}
			\vspace*{-0.15in}
			\begin{center}
				\begin{tabular}{llllll}
					\hline
					Algorithm & dataset & $\epsilon=10^{-4}$ & $\epsilon=10^{-5}$ & $\epsilon=10^{-6}$ & $\epsilon=10^{-7}$\\ 
					\hline
					PG    & bodyfat & 366637 & 1110329 & 1871925 & 1948897 \\
					adaAGC & bodyfat & 15414 & 26174 & 40526 & 40905\\
					\midrule
					PG & cpusmall & 109298 & 159908 & 170915 & 170915\\
					adaAGC & cpusmall & 9571  & 12623  & 13571  & 13571\\\hline
				\end{tabular}
			\end{center}
		
			\caption{Huber loss with $\ell_1$ norm regularization}
			\label{exp3}
			\vspace*{-0.15in}
			\begin{center}
					\begin{tabular}{llllll}
						\hline
						Algorithm & dataset & $\epsilon=10^{-4}$ & $\epsilon=10^{-5}$ & $\epsilon=10^{-6}$ & $\epsilon=10^{-7}$\\ 
						\hline
						PG    & bodyfat & 258723 & 423181 & 602043 & 681488 \\
						adaAGC & bodyfat & 16976 & 16980 & 23844 & 25702\\
						\midrule
						PG & cpusmall & 74387 & 112702 & 159461 & 190640\\
						adaAGC & cpusmall & 26958  & 32070 & 36698  & 38222\\\hline
					\end{tabular}
				\end{center}
				\end{table}
				\begin{table}[t]
				\caption{$\ell_1$ constrained $\ell_p$ norm regression on bodyfat ($\epsilon=10^{-3}$)}\label{exp4}
				\begin{center}
						\begin{tabular}{lllll}
							\hline
							Algorithm & $p=2$ & $p=4$ & $p=6$ & $p=8$\\ 
							\hline
							PG & 250869 (1) & 979401 (3.90) & 1559753 (6.22) & 4015665 (16.00)\\
							adaAGC & 8710 (1) & 17494 (2.0) & 22481 (2.58) & 33081 (3.80)\\\hline
						\end{tabular}
					\end{center}
				\end{table}

		We use the backtracking in both PG and adaAGC to search for the smoothness parameter. 
		In adaAGC, we set $c_0=10,\gamma=2$. 
		For fairness, each algorithm starts at the same initial point, which is set to be zero, and we stop each algorithm when the norm of its proximal gradient is less than a prescribed threshold $\epsilon$ and {\bf report the total number of proximal mappings}. 
		The results are presented in the Tables~\ref{exp1},~\ref{exp2},~\ref{exp3},~\ref{exp4},~\ref{exp21},~\ref{exp22} and \ref{exp23}, which clearly show that adaAGC converges considerably faster than PG. It is notable that for some problems (see Table~\ref{exp1},~\ref{exp21})  the number of proximal mappings is the same value for achieving different precision $\epsilon$. This is because that value is the minimum number of proximal mappings such that the magnitude of the proximal gradient suddenly becomes zero. In Table~\ref{exp4}, the numbers in parenthesis indicate the increasing factor in the number of proximal mappings compared to the base case $p=2$, which show that increasing factors of adaAGC are approximately the square root of that of PG and hence are consistent with our theory.


		\begin{table}[t]
					\caption{Squared hinge loss with $\ell_\infty$ regularization}
					\label{exp21}
					\begin{center}
						\begin{tabular}{llllll}
							\hline
							Algorithm & dataset & $\epsilon=10^{-4}$ & $\epsilon=10^{-5}$ & $\epsilon=10^{-6}$ & $\epsilon=10^{-7}$\\ 
							\hline
							PG    & splice & 3514 & 3724 & 3724 & 3724 \\
							adaAGC & splice & 2336 & 2456 & 2456 & 2456\\
							\midrule
							PG & german.numer & 898 & 898 & 898 & 898\\
							adaAGC & german.numer & 742  & 742  & 742  & 742\\\hline
						\end{tabular}
					\end{center}
					

					\caption{Squared loss with $\ell_\infty$ regularization}
					\label{exp22}
								\vspace*{-0.15in}

					\begin{center}
						\begin{tabular}{llllll}
							\hline
							Algorithm & dataset & $\epsilon=10^{-4}$ & $\epsilon=10^{-5}$ & $\epsilon=10^{-6}$ & $\epsilon=10^{-7}$\\ 
							\hline
							PG    & bodyfat & 542414 & 652613 & 778869 & 800050 \\
							adaAGC & bodyfat & 23226 & 24990 & 30646 & 30864\\
							\midrule
							PG & cpusmall & 139505 & 204120 & 210874 & 210874\\
							adaAGC & cpusmall & 10828  & 14276  & 15020  & 15020\\\hline
						\end{tabular}
					\end{center}

					\caption{Huber loss with $\ell_\infty$ regularization}
					\label{exp23}
			\vspace*{-0.15in}
					\begin{center}
						\begin{tabular}{llllll}
							\hline
							Algorithm & dataset & $\epsilon=10^{-4}$ & $\epsilon=10^{-5}$ & $\epsilon=10^{-6}$ & $\epsilon=10^{-7}$\\ 
							\hline
							PG    & bodyfat & 419316 & 531999 & 651092 & 709486 \\
							adaAGC & bodyfat & 15744 & 18072 & 23684 & 25391\\
							\midrule
							PG & cpusmall & 75346 & 171052 & 240050 & 270540\\
							adaAGC & cpusmall & 27225  & 36745  & 41461  & 42925\\\hline
						\end{tabular}
					\end{center}
				\end{table}

				\section{Conclusions}

				In this paper, we have considered smooth composite optimization problems under a general H\"{o}lderian error bound condition. We have established adaptive iteration complexity to the  H\"{o}lderian error bound condition of proximal gradient and accelerated proximal gradient methods. To eliminate the dependence on the unknown parameter in the  error bound condition and enjoy the faster convergence of accelerated proximal gradient method, we have developed a parameter-free adaptive accelerated gradient converging method using the magnitude of the (proximal) gradient as a measure for restart and termination. We have also considered a broad family of norm regularized problems in machine learning and showed faster convergence of the proposed adaptive accelerated gradient converging method.

\section*{Appendix}
We present some lemmas and propositions that are useful to our analysis.
\begin{prop}\citep[Theorem 5 in v3]{arxiv:1510.08234}
	\label{prop_bolte}
	Let $f:H\rightarrow (-\infty,+\infty]$ be a proper, convex and lower semi-continuous with $\min f = f_*$. Let $r_0>0$, $\varphi\in\{\varphi\in C^0[0,r_0)\cap C^1(0,r_0),\varphi(0)=0,\varphi \text{ is concave}, \varphi>0\}$, $c>0,\rho>0$, and $\bar{x}\in \arg\min f$. If $s\varphi'(s)\geq c\varphi(s)$ for all $s\in(0,r_0)$, and $\varphi(f(x) - f_*)\geq D(x,\arg\min f)$ for all $x\in[0<f<r_0]\cap B(\bar{x},\rho)$, then $\varphi'(f(x)-f_*)\Vert \partial f(x)\Vert_2\geq c$ for all $x\in[0<f<r_0]\cap B(\bar{x},\rho)$.
	\end{prop}
	The following proposition is a rephrase of Theorem 3.5 in~\citep{Drusvyatskiy16a}. 
	\begin{prop}
		\label{prop_lewis}
		If $f$ is $L$-smooth and convex, $g$ is proper, convex and lower semi-continuous, $F(\x)=f(\x)+g(\x)$, $\eta>0$, and define
		\begin{equation*}
		P_{\eta F}(\x)=\arg\min\limits_{\u}\frac{1}{2}\Vert \u-\x\Vert_2^2+\eta F(\u).
		\end{equation*}
		Then the following inequality holds:
		\begin{equation*}
		\left\Vert\frac{1}{\eta}(\x-P_{\eta F}(\x))\right\Vert_2\leq(1+L\eta)\Vert G_\eta(\x)\Vert_2.
		\end{equation*}
		\end{prop}
\subsection{Perturbation of a Strongly Convex Problem}
\begin{lemma}
	\label{strongly}
	Let $h(\x)$ be a $\sigma$-strongly convex function, $\x_a^*$ and $\x_b^*$ be the optimal solutions to the following problems. 
	\begin{align*}
	\x_a^* = \min_{\x\in\R^d}\a^{\top}\x + h(\x).
	\end{align*}
	\begin{align*}
	\x_b^* = \min_{\x\in\R^d}\b^{\top}\x + h(\x).
	\end{align*}
	Then
	\[
	\|\x_a^* - \x_b^*\|_2 \leq \frac{2\|\a - \b\|_2}{\sigma}.
	\]
\end{lemma}
\begin{proof}
	Let $H_a(\x) = h(\x) + \a^{\top}\x$ and  $H_b(\x) = h(\x) + b^{\top}\x$. By the strong convexity of $h(\x)$, we have 
	\begin{align*}
	&\frac{\sigma}{2}\|\x_a^* - \x_b^*\|_2^2\\
	&\leq H_a(\x_b^*) - H_a(\x_a^*)\\
	&=H_b(\x_b^*) + (\a - \b)^{\top}\x_b^*- H_b(\x_a^*) -(\a - \b)^{\top}\x_a^* \\
	& \leq (\a - \b)^{\top}(\x_b^* - \x_a^*)\leq \|\x_a^* -\x_b^*\|_2 \|\a - \b\|_2,
	\end{align*}
	where we use the fact $H_b(\x_b^*)\leq H_b(\x_a^*)$. From the above inequality, we can get $\|\x_a^* - \x_b^*\|_2\leq \frac{2\|\a- \b\|_2}{\sigma}.$
\end{proof}
\subsection{Proof of Lemma \ref{lem:4}}
\begin{proof}
	The conclusion is trivial if $\x\in\Omega_*$. Otherwise, the proof follows Proposition~\ref{prop_bolte}. In particular, if we define $\varphi(s)=cs^{\theta}$, then $D(\x, \Omega_*)\leq \varphi(f(\x) - f_*)$ for any $\x\in\{\x: 0<f(\x) - f_*\leq\xi\}$ and $\varphi$ satisfies  $s\varphi'(s)\geq \theta \varphi(s)$. By Proposition~\ref{prop_bolte}, we have 
	\begin{align*}
	\varphi'(f(\x) - f_*)\|\partial f(\x)\|_2\geq \theta,
	\end{align*}
	i.e., 
	\begin{align}
	\label{KL}
	c\|\partial f(\x)\|_2 \geq (f(\x) - f_*)^{1-\theta}.
	\end{align}
	When $\theta=1$, we have $\|\partial f(\x)\|_2\geq 1/c$ for $\x\not\in\Omega_*$. 
	As a result, when $\theta\in(0, 1)$
	\begin{align*}
	D(\x, \Omega_*)\leq c(f(\x) - f_*)^{\theta}\leq c^{\frac{1}{1-\theta}}\|\partial f(\x)\|_2^{\frac{\theta}{1-\theta}},
	\end{align*}
	and when $\theta=1$
	\begin{align*}
	D(\x, \Omega_*)\leq c(f(\x) - f_*)\leq c^2\xi\|\partial f(\x)\|_2.
	\end{align*}
	
	\end{proof}
\bibliography{all}

\end{document}